\documentclass[11pt]{amsart}

\usepackage{amsmath, amsfonts, amssymb,amsthm}
\usepackage{amstext}
\usepackage{mathrsfs}
\usepackage{graphicx}
\usepackage{float,epsf,subfigure}
\usepackage[all,cmtip]{xy}
\usepackage{hyperref}
\usepackage{algorithm}
\usepackage{algorithmic}
\usepackage{mathtools}
\usepackage{float}
\usepackage{bm}
\usepackage{graphicx}
\theoremstyle{plain}
\newtheorem{theorem}{Theorem}[section]

\newtheorem{cor}[theorem]{Corollary}

\newtheorem{def-thm}[theorem]{Definition-Theorem}
\newtheorem{lemma}[theorem]{Lemma}

\newtheorem{defi}[theorem]{Definition}

\newtheorem*{tha}{Theorem A}

\newtheorem{remark}{Remark}

\newtheorem{property}[theorem]{Property}

\theoremstyle{definition}
\newtheorem{example}[theorem]{Example}
\newtheorem{question}{Question}

\def\min{\mathop{\mathrm{min}}}

\begin{document}
\title[Value distribution of meromorphic mappings]{Value distribution of meromorphic mappings  on complete K\"ahler connected sums with non-parabolic ends}
\author[X.-J. Dong]
{Xianjing Dong}

\address{School of Mathematical Sciences \\ Qufu Normal University \\ Qufu, Jining, Shandong, 273165, P. R. China}
\email{xjdong05@126.com}


\date{}
\maketitle \thispagestyle{empty} \setcounter{page}{1}

\begin{abstract} 
All   harmonic functions on $\mathbb C^m$ possess  Liouville's property, which is  well-known as the Liouville's theorem. 
In 1979, Kuz'menko and Molchanov discovered   a phenomenon  that the  Liouville's property is not rigid for some  harmonic functions on 
the connected sum $\mathbb C^m\#\mathbb C^m,$ 
where  there exist   a large number of   non-constant bounded harmonic functions.  
This  discovery motivates  us  to explore    conditions 
under which  harmonic functions  possess  Liouville's property.  
In this paper,  
  we discuss    the    value distribution      of     meromorphic mappings from  complete K\"ahler connected sums with non-parabolic ends into  complex projective manifolds. 
   Under a  geometric  condition, 
  we establish  a second main theorem in Nevanlinna theory. 
  As a  consequence,   we  prove     that  the   Cauchy-Riemann   equation ensures      
            the   rigidity  of   Liouville's property  for harmonic 
        functions 
      if  such connected sums satisfy  a volume growth condition.

 \end{abstract} 
 
\vskip\baselineskip

\vskip\baselineskip

\tableofcontents
\newpage
\setlength\arraycolsep{2pt}
\medskip

\section{Introduction}
\vskip\baselineskip

\subsection{Motivation}~
 
Each  bounded harmonic function on $\mathbb C^m$ is  a constant, which is well-known as  the  Liouville's theorem or Liouville's property. 
However,  it doesn't 
   always hold   for a harmonic function    on the connected sum  $\mathbb C^m\#\mathbb C^m$ which supports   a K\"ahler structure.  
   As early  as 1979, Kuz'menko-Molchanov \cite{M-K} considered the 
 connected sum $\mathbb R^m\#\mathbb R^m,$  as an example of    manifolds (see Fig. 1),  
  where the Liouville's property fails to some  bounded harmonic functions. 
       Later in 1999, 
        Grigor'yan-Saloff-Coste (see \cite{Gri6, Gri}) also observed  this      phenomenon  
  on  a more general connected sum of complete weighted manifolds 
    with non-negative Ricci curvature,
       on which  they  constructed a non-constant bounded harmonic function in terms of  the heat kernel.
          In order to serve  the present  paper,  
                    we  would like to  provide    a  counterexample  to the  Liouville's property on $\mathbb C^m\#\mathbb C^m,$  
                                  which is     a special case of  Grigor'yan-Saloff-Coste's theorem   
  (see \cite{Gri}, Proposition 6.3). 
  
      \begin{figure}[htbp]
  \centering
  \includegraphics[width=13cm]{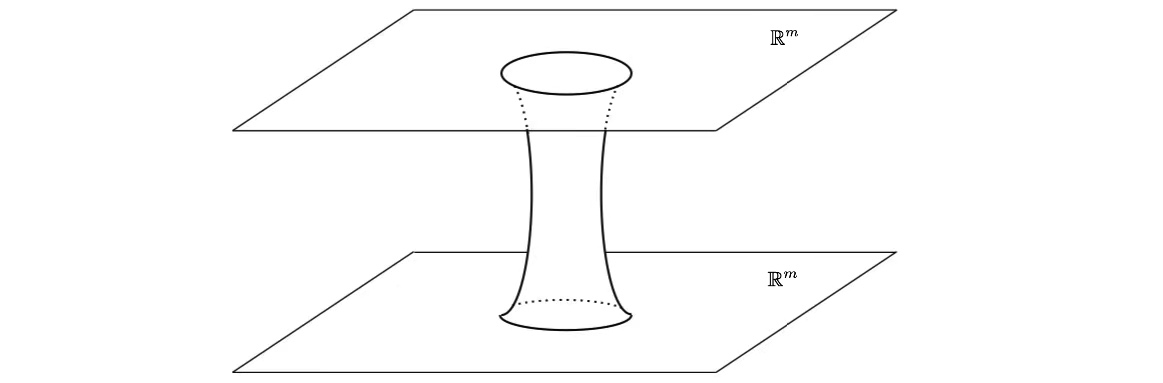}
  \caption{$\mathbb R^m\#\mathbb R^m$} 
 \end{figure} 

Let  $\#^\vartheta\mathbb C^m$ denote   the K\"ahler  connected sum of $\vartheta$ copies of $\mathbb C^m.$ 
 Precisely,  $\#^\vartheta\mathbb C^m$ is a K\"ahler  manifold which  satisfies  that  
  there exists  a  compact subset $K\subset \#^\vartheta\mathbb C^m$ (with nonempty interior) 
    such that  $\#^\vartheta\mathbb C^m\setminus K$ is a disjoint union 
    of  open subsets $E_1,\cdots, E_\vartheta,$ here  
     each $E_j$ is 
          analytically isometric to 
          $\mathbb C^m\setminus K_j$ for a compact subset $K_j\subset\mathbb C^m.$ 
      Each $E_j$  can  be  identified with $\mathbb C^m\setminus K_j.$ 
 We   refer to  $K$   as  the \emph{central part} of 
$\#^\vartheta\mathbb C^m,$  and refer  to $E_1,\cdots, E_\vartheta$ as  the \emph{ends} of $\#^\vartheta\mathbb C^m.$
Note that  $\mathbb C^m=\#^\vartheta\mathbb C^m$  with $\vartheta=1.$ 
For  $x\in\#^\vartheta\mathbb C^m,$ set 
$$|x|=\sup_{y\in K}d(x,y),$$
where  $d(x,y)$ is  the Riemannian distance between  $x, y.$ 
  Fix a reference point $o$ in the interior of $K.$ Let   
  $B(r)$  be the  geodesic ball  centered at $o$ with radius $r$ in $ \#^\vartheta\mathbb C^m.$ Put    
$$V_{0}(r)=\min_{1\leq j\leq\vartheta}V_j(r),$$
where   
$$V_j(r)={\rm{Vol}}\left(B(r)\cap (K\cup E_j)\right), \ \ \ \  j=1,\cdots,\vartheta.$$
Again, put 
  $$ i_x=\begin{cases}
0, \   &   x\in K; \\
j,  \  &   x\in E_j \ \text{for some} \ j.
\end{cases}$$
\ \ \ \ Let $u, v$ be two non-negative  functions defined on the same domain. We then write 
$u\simeq v,$ if there are constants $\alpha,\beta>0$ such that $\alpha v\leq u\leq\beta v.$ 
Put $w^+=\max\{w, 0\},$ i.e.,  the positive part of a real function $w.$
Let  $\mathscr H(\#^\vartheta\mathbb C^m)$ denote  the space  of   harmonic  functions on $\#^\vartheta\mathbb C^m,$ and 
 $\mathscr H^\flat(\#^\vartheta\mathbb C^m)$ denote  the space  of bounded   harmonic  functions on $\#^\vartheta\mathbb C^m,$ 
 and $\mathscr H^+(\#^\vartheta\mathbb C^m)$ denote the linear space spanned by the set of 
  positive harmonic functions on $\#^\vartheta\mathbb C^m.$
 
  The following theorem is due to   Grigor'yan-Saloff-Coste \cite{Gri},  Li-Tam\cite{L-T1} and  Sung-Tam-Wang \cite{S-T-W},  as  a  counterexample to the Liouville's property.  

\begin{tha}  There exists a positive  harmonic function $h\in \mathscr H(\#^\vartheta\mathbb C^m)$ such that  
$$h(x)\simeq 1+\left(\int_1^{|x|}\frac{dt}{V_{i_x}(\sqrt t)}\right)^+.$$
Hence, the Liouville's property is not rigid for $ \mathscr H(\#^\vartheta\mathbb C^m).$
\end{tha}

 In fact,  Li-Tam  \cite{L-T1, L-T2}  proved          that the dimension  of the space of bounded harmonic functions 
          is related   to   the number of ends of a manifold. According to  
                  their theorem  (see \cite{L-T1}, Theorem 7.3), we  have    
$$\dim\mathscr H^\flat(\#^\vartheta\mathbb C^m)=\dim\mathscr H^+(\#^\vartheta\mathbb C^m)=\vartheta.$$
\ \ \    \    More generally,   given  a   K\"ahler  connected sum  $M=M_1\#\cdots\# M_\vartheta,$ where $M_1,\cdots, M_\vartheta$ are 
 complete  non-compact K\"ahler manifolds with non-negative Ricci curvature.
  It is natural  to ask that under  which   conditions    a   harmonic function  on $M$  possesses        Liouville's property. 
  It is noticed    that  the  harmonic functions  (as a 
            counterexample to the Liouville's property) on $\#^\vartheta\mathbb C^m$  do not  satisfy    Cauchy-Riemann  equation. 
            So,  it     inspires us to wonder    
                                  whether  
                              the  Cauchy-Riemann equation maintains     the rigidity of    Liouville's property  for  harmonic functions on $M.$
  Let  $\mathscr O(M)$ be  the space  of   holomorphic  functions on $M.$  We propose the following question. 

\begin{question}
\emph{Is  Liouville's property rigid   for} $\mathscr O(M)$?
\end{question}

    Picard's little theorem asserts   that every meromorphic function on $\mathbb C^m$   is    a  constant  if it omits three distinct values, which   generalizes   the   Liouville's  theorem.  
   Let $\mathscr M(M)$ be the space   of meromorphic functions on $M.$
  Further, we  propose   the following  question. 

\begin{question}
\emph{Does  Picard's little theorem hold   for}  $\mathscr M(M)$?
\end{question}
The main purpose of this   paper is to answer the above questions. In doing so, 
we  shall  generalize    Nevanlinna theory \cite{Nev} to   complete K\"ahler connected sums with 
non-parabolic ends using  the global Green function approach. As a consequence of our main theorem, 
we  show   that the Picard's little theorem holds   for $\mathscr M(M)$ when  $M$ satisfies a volume growth condition. 
 In particular, the 
Cauchy-Riemann  equation   maintains   the  rigidity of   Liouville's property  for   $\mathscr O(M)$  under this condition.

\subsection{Main Results}~
 
We  give a brief introduction to  the main results in this  paper. 
 Let 
       $$M=M_1\#\cdots\# M_\vartheta=\left(K; \{M_j\}_{j=1}^\vartheta; \{E_j\}_{j=1}^\vartheta\right)$$
be a complete K\"ahler connected sum with $\vartheta$ non-parabolic  ends, where  all  
      $M_j^,s$ are complete non-compact  K\"ahler manifolds satisfying   $(PH)$  (see  Property \ref{ap2}  in Section 2.2 and Definition \ref{defcc} in Section 3.2), with Ricci curvature bounded from below by a constant.
       Note that $E_j^,s$ are called  the ends of $M$ and $K$ is called  the central part of $M$ (see definition  in Section 2.3). 
               Fix a reference point $o$ in the  interior of $K.$
                               Let  $\rho(x)$  be  the Riemannian distance function of $x$ from $o,$ and  $V(r)$ be  the Riemannian volume of  a geodesic ball $B(r)$ centered at $o$ with radius $r$ in $M.$ Set 
$$V_{\rm min}(r)=\min_{1\leq j\leq\vartheta}V_j(r), \ \ \ \   V_{\max}(r)=\max_{1\leq j\leq\vartheta}V_j(r)$$
with 
$$V_j(r)={\rm{Vol}}\left(B(r)\cap (K\cup E_j)\right), \ \ \ \  j=1,\cdots,\vartheta.$$
\ \ \ \ Let $X$ be a complex projective manifold of complex dimension not greater than that of $M.$  We  can put a positive line bundle $L$ over $X.$ Let 
$f: M\to X$ be a meromorphic mapping. 
 With the two-sided bounds of the heat kernel of $M,$ 
we  can construct  a family $\{\Delta(r)\}_{r>0}$ of exhaustive precompact domains  for $M$ (see definition in Section 3.3). 
Taking  a divisor  $D\in|L|.$ 
On $\Delta(r),$ we can define   Nevanlinna's functions  $T_f(r, L), m_f(r, D),  \overline{N}_f(r,D)$ and $T(r, \mathscr R),$ where $\mathscr R:=-dd^c\log\det(g_{i\bar j})$ denotes  the Ricci form associated to the metric  $g=(g_{i\bar j})$  of $M$ (see details  in Section 3.3). 
Let 
\begin{equation}\label{min1}
\kappa:=\min_{1\leq j\leq\vartheta}\kappa_j,
\end{equation}
 where $\kappa_j$ is  the lower bound  of Ricci curvature of $M_j$ for $j=1,\cdots,\vartheta.$
Again, set
 \begin{equation}\label{XI}
 \Xi(r, \delta, \kappa)=\frac{\left(|\kappa|+r^{-1}\right)\displaystyle\int_0^\infty\frac{e^{-\frac{r^2}{bt}}}{V_{\rm min}(\sqrt{t})}dt}{\left(r\displaystyle\int_0^\infty\frac{e^{-\frac{r^2}{at}}}{V_{\max}(\sqrt{t})}\frac{dt}{t}\right)^{1+\delta}}.
  \end{equation}
  
  We establish  a   second main theorem in Nevanlinna theory  as follows.  
     \begin{theorem}\label{main1}  Let $M=M_1\#\cdots\#M_\vartheta$ be a  
     complete K\"ahler connected sum with $\vartheta$  non-parabolic  ends, 
     where  all  $M_j^{,}s$ are   non-compact manifolds  satisfying   $(PH)$  in Property $\ref{ap2},$ with Ricci curvature bounded from below by a constant. 
     Let $X$ be a complex projective manifold of complex dimension not greater than  that  of $M.$
 Let $D\in|L|$ be a reduced divisor of simple normal crossing type, where $L$ is a positive line bundle over $X.$   Let  $f:M\rightarrow X$ be a  differentiably non-degenerate meromorphic mapping.    Then  for any  $\delta>0,$ there exists a subset $E_\delta\subset(0, \infty)$ of finite Lebesgue measure such that 
        \begin{eqnarray*}
&&T_f(r,L)+T_f(r, K_X)+T(r, \mathscr R) \\
&\leq& \overline N_f(r,D)+O\left(\log^+ T_{f}(r,L)+\log^+\Xi(r,\delta, \kappa)\right) 
         \end{eqnarray*}
    holds for all $r>0$ outside $E_\delta,$      where $\Xi(r,\delta,\kappa)$ is defined by $(\ref{XI}).$
\end{theorem}  
  
By estimating $\Xi(r,\delta,\kappa)$ using the volume comparison theorem,  we obtain: 

       \begin{theorem}\label{main2}  Let $M=M_1\#\cdots\#M_\vartheta$ be a  
     complete K\"ahler connected sum with $\vartheta$  non-parabolic  ends, 
   where  all  $M_j^{,}s$ are   non-compact manifolds  satisfying   $(PH)$  in Property $\ref{ap2},$ with Ricci curvature bounded from below by a constant. 
        Let $X$ be a complex projective manifold of complex dimension not greater than  that  of $M.$
 Let $D\in|L|$ be a reduced divisor of simple normal crossing type, where $L$ is a positive line bundle over $X.$   Let  $f:M\rightarrow X$ be a  differentiably non-degenerate meromorphic mapping.    Then  
        \begin{eqnarray*}
&&T_f(r,L)+T_f(r, K_X)+T(r, \mathscr R) \\
&\leq& \overline N_f(r,D)+O\left(\log^+ T_{f}(r,L)+r^2\right) 
         \end{eqnarray*}
    holds for all $r>0$ outside a subset  of finite Lebesgue measure.  
    \end{theorem}

  By Li-Yau's argument \cite{L-Y}, the non-negativeness of  Ricci curvature of all $M_j^,s$ implies the  two-sided heat kernel bounds  
$(HK)$ in Property \ref{ap2}.  Note that $(HK)$  is equivalent to $(PH)$  due to  Theorem  \ref{equi2} in Section 2.2. 
When all $M_j^,s$ have non-negative Ricci curvature, we show that 
  
    \begin{theorem}\label{main3}  Let   $M=M_1\#\cdots\#M_\vartheta$ be a  complete K\"ahler connected sum with $\vartheta$  non-parabolic  ends, where  all  $M_j^,s$ are non-compact   manifolds with non-negative Ricci curvature.    
Let $X$ be a complex projective manifold of complex dimension not greater than  that  of $M.$
 Let $D\in|L|$ be a reduced divisor of simple normal crossing type, where $L$ is a positive line bundle over $X.$  Let  $f:M\rightarrow X$ be a  differentiably non-degenerate meromorphic mapping.   Then  
      \begin{eqnarray*}
&&T_f(r,L)+T_f(r, K_X)+T(r, \mathscr R) \\
&\leq& \overline N_f(r,D)+O\left(\log^+ T_{f}(r,L)+\log r\right) 
         \end{eqnarray*}
    holds for all $r>0$ outside a subset  of finite Lebesgue measure.  
\end{theorem}
  
  Grigor'yan-Saloff-Coste   introduced the concept  of \emph{homogeneousness}  for a connected sum  as follows  (see \cite{Gri}, p. 1984). 
  \begin{defi}\label{ah}
   We say   that $M$ is  homogeneous,  if   
 $$V_{i}(r)\simeq V(r)\simeq V_{j}(r)$$
holds  for  all $i,j=1,\cdots,\vartheta.$
   \end{defi}
   
      Set
 \begin{equation}\label{X}
 E(r)=\frac{V(r)}{r^{2}}\int_r^\infty\frac{tdt}{V(t)}.
  \end{equation}

  When  $M$ is homogeneous  and all $M_j^,s$ have non-negative Ricci curvature,  
       we   show that  

  \begin{theorem}\label{main4}  Let   $M=M_1\#\cdots\#M_\vartheta$ be a homogeneous complete K\"ahler connected sum with $\vartheta$  non-parabolic  ends, where  all  $M_j^,s$ are non-compact   manifolds with non-negative Ricci curvature.    
Let $X$ be a complex projective manifold of complex dimension not greater than  that  of $M.$
 Let $D\in|L|$ be a reduced divisor of simple normal crossing type, where $L$ is a positive line bundle over $X.$  Let  $f:M\rightarrow X$ be a  differentiably non-degenerate meromorphic mapping.    Then  for any  $\delta>0,$ there exists a subset $E_\delta\subset(0, \infty)$ of finite Lebesgue measure such that 
      \begin{eqnarray*}
&&T_f(r,L)+T_f(r, K_X)+T(r, \mathscr R) \\
&\leq& \overline N_f(r,D)+O\left(\log^+ T_{f}(r,L)+\log^+E(r)+\delta\log r\right) 
         \end{eqnarray*}
holds for  all $r>0$ outside $E_\delta,$ where $E(r)$ is defined by $(\ref{X}).$   
\end{theorem}

 We consider some defect relations. The    \emph{simple defect}  $\bar\delta_f(D)$ and  notation $[c_1(K_X^*)/c_1(L)]$ (see Carlson-Griffiths \cite{gri})
  are   defined respectively  by
 \begin{eqnarray*}
 \bar\delta_f(D)&=&1-\limsup_{r\rightarrow\infty}\frac{\overline{N}_f(r,D)}{T_f(r,L)}
 \end{eqnarray*}
 and
$$\left[\frac{c_1(K_X^*)}{c_1(L)}\right]=\inf\left\{s\in\mathbb R: \ \eta\leq s\omega;  \  \  ^\exists\omega\in c_1(L),\  ^\exists\eta\in c_1(K^*_X) \right\}.$$
 
As some applications  of  Theorems \ref{main1}, \ref{main2} and \ref{main3},  we obtain the following  defect relations in Nevanlinna theory. 

\begin{cor}  Assume the same conditions as in Theorem  $\ref{main2}.$ If 
$$\lim_{r\to\infty}\frac{r^2}{T_f(r, L)}=0,$$
then 
$$\bar\delta_f(D)
\leq \left[\frac{c_1(K_X^*)}{c_1(L)}\right].$$ 
\end{cor}

\begin{cor}  Assume the same conditions as in Theorem  $\ref{main3}.$ If 
$$\lim_{r\to\infty}\frac{\log r}{T_f(r, L)}=0,$$
then 
$$\bar\delta_f(D)
\leq \left[\frac{c_1(K_X^*)}{c_1(L)}\right].$$ 
\end{cor}

\begin{cor}\label{corm}  Assume the same conditions as in Theorem  $\ref{main4}.$ If
\begin{equation}\label{jy}
\lim_{r\to\infty}\frac{\log^+E(r)}{\log r}=0,
\end{equation}
then 
$$\bar\delta_f(D)
\leq \left[\frac{c_1(K_X^*)}{c_1(L)}\right].$$ 
\end{cor}

We conclude   from Corollary \ref{corm} that

\begin{cor}\label{corh}    Picard's little theorem  holds   for  $\mathscr M(M)$ if $(\ref{jy})$ is satisfied.   In particular,  the Liouville's property is rigid  for $\mathscr H(M)$  if $(\ref{jy})$ is satisfied.
\end{cor}

In particular, we are interested in the case when $M=\#^\vartheta\mathbb C^m.$ It is known  that  $\mathbb C^m$ is non-parabolic for $m\geq2.$
By $V(r)=O(r^{2m})$ as $r\to\infty,$ we obtain  
$$\lim_{r\to\infty}\frac{\log^+E(r)}{\log r}=0.$$
  
   \section{Preliminaries}

For the reader's convenience,  we   outline  some facts from  geometric analysis including connected sums of Riemannian manifolds and their heat kernel estimate (see \cite{Gri6, Gri, L-Y, Sa, S-Y}).  Throughout the paper, unless otherwise specified, a  manifold always   means  boundless.  

 \subsection{Weighted Manifolds}~

Let $(M, g)$ be a Riemannian manifold with boundary $\delta M$ (which may be empty).  
 Fix a  positive smooth  function $\sigma$ on $M.$ We  define a measure $\mu$ on $M$ 
by  
$$d\mu=\sigma^2dv,$$
where $dv$ is the Riemannian measure (i.e., the Riemannian volume element)  on $M.$   The pair $(M, \mu)=(M, g, \mu)$  is called  a \emph{weighted manifold.}     
 Let  $B(x, r)$ denote   the geodesic ball centered at $x$ with radius $r$ in $M.$ Set 
 $$V(x, r)=\mu(B(x, r)), $$
 which is called the \emph{$\mu$-volume} of $B(x, r).$ Evidently, $V(x,r)$ is the Riemannian volume of $B(x, r)$ when 
 $\sigma\equiv1.$
 We say  that $M$ is complete if  the metric  space $(M, d)$ is complete,  here   $d$ is the Riemannian distance on $M$ induced by the metric $g.$ 
  When $\delta M=\emptyset,$  
  a well-known  fact states    that $M$ is complete if and only if $M$ is geodesically complete (i.e., the maximal 
defining interval of any geodesic  in $M$ is $\mathbb R$).

  For any smooth function $u$  on $M,$ the \emph{gradient} $\nabla u$   is a unique vector field defined by $g(\nabla u, X)=du(X)$ for each smooth vector field $X$ on $M.$ Locally,  $\nabla u$   can  be  explicitly determined  by 
 $$(\nabla u)^j=\sum_{i} g^{ij}\frac{\partial u}{\partial x_i}$$
 in  local coordinates  $x_1,\cdots, x_m,$ 
  where $(g^{ij})$ is  the inverse  of metric matrix  $g=(g_{ij}).$ It  gives     
  $$\|\nabla u\|^2=g\left(\nabla u, \nabla u\right)=\sum_{i,j}g^{ij}\frac{\partial u}{\partial x_i} \frac{\partial u}{\partial x_j}.$$
  The \emph{divergence} of a smooth vector field $X$  is defined by  
  $${\rm div}X=\frac{1}{\sqrt{G}}\sum_{j}\frac{\partial}{\partial x_j}\left(\sqrt{G}X^j\right),$$
 where $G=\det(g_{ij}).$ Whence, we have  the \emph{Laplace-Beltrami operator} 
  $$\Delta:={\rm{div}}\circ\nabla=\frac{1}{\sqrt{G}}\sum_{i,j}\frac{\partial u}{\partial x_i}\left(\sqrt{G}g^{ij}\frac{\partial}{\partial x_j}\right).$$
 With a weight $\mu,$     we also define the \emph{weighted divergence}   of $X$ by  
 $${\rm{div}}_\mu X=\frac{1}{\sigma^2\sqrt{G}}\sum_{j}\frac{\partial}{\partial x_j}\left(\sigma^2\sqrt{G}X^j\right).$$
 Note that  ${\rm{div}}_\mu={\rm{div}}$ if $\sigma\equiv1.$  
  It   induces  the \emph{weighted Laplace operator}      
  $$L:={\rm{div}_\mu}\circ\nabla=\sigma^{-2}{\rm{div}}\circ(\sigma^2\nabla).$$
Clearly,  we have   $L=\Delta$  if $\sigma\equiv1.$
  
  Consider the Dirichlet form 
  $$\mathscr D(u, v)=\int_Mg\left(\nabla u, \nabla v\right)d\mu$$
  for  all $u,v\in\mathscr C^\infty_c(M)$ (the space of  smooth functions with  compact support  on $M$). Using the integration-by-parts formula, we have 
  $$-\mathscr D(u,v)=\int_Mu Lvd\mu+\int_{\delta M}u\frac{\partial u}{\partial{\vec{\nu}}}d\mu',$$
  where $\partial/\partial{\vec{\nu}}$ is the inward normal derivative on $\delta M,$ and $\mu'$ is the measure with density $\sigma^2$ with respect to the Riemannian measure of codimension 1 on any hypersurface of $M,$  in particular,  on $\delta M.$ Note that $ L$ is symmetric on the subspace of $\mathscr C^\infty_c(M)$ of functions with vanishing derivative on $\delta M.$ It yields  that $ L$ initially defined on this subspace admits
   a Friedrichs extension 
   $\overline{ L}$ that is a self-adjoint and non-positive definite operator on $L^2(M,\mu).$
But,   $\Delta$ is indeed a self-adjoint and semi-positive definite operator. 
The associated heat semigroup $P_t=e^{t\overline{ L}}$    has  the smooth integral kernel, say $p(t,x,y),$ 
   which  is called the \emph{heat kernel} of $(M, \mu)$ associated to $\overline{ L}.$ 
    Alternatively, $p(t, x,y)$ is  
     the minimal positive fundamental solution of the Cauchy problem
      $$\begin{cases}
\left( L-\frac{\partial}{\partial t}\right)u(t,x)=0,  \ \ &   (t, x)\in (0,\infty)\times M; \\
u(0, x)=\delta_y(x), \ \  &  x, y\in M; \\
\frac{\partial u(t,x)}{\partial{\vec{\nu}}}\big{|}_{\delta M}=0, \ \ &  0<t<\infty,
\end{cases}$$
 where $\delta_x$ is the Dirac's delta  function on $M$ (see \cite{Ch, Do, Ro}). Note that  
 $$p(t, x,y)=p(t, y, x).$$
  The heat kernel $p(t, x,y)$ carries   a probabilistic meaning. Let $X_t:=(X_t)_{t\geq0}$ be a heat  diffusion process which is generated by
  $\overline{ L}$ on $M.$  Let $\mathbb P_x$ be the law  of $X_t$ starting from $x\in M.$ 
 Then,  $p(t,x,y)$ represents    the transition density  of $X_t,$ that is, 
  $$\mathbb P_x(X_t\in A)=\int_A p(t,x,y)d\mu(y)$$
  for any Borel set $A\subset M.$ 
  It is noted   that the Neumann boundary condition implies  that  $X_t$ is reflected on $\delta M.$
  We say that $(M, \mu)$ is \emph{non-parabolic} if
  $$\int_1^\infty p(t,x,y)dt<\infty$$
  for some (and hence for all) $x, y\in M,$ and \emph{parabolic} otherwise. A well-known fact claims  that the parabolicity of $M$ is equivalent to the recurrence of $X_t$ (see  \cite{Gri1}).

   \subsection{Faber-Krahn  Inequality and  Harnack Inequality}~

   \subsubsection{Faber-Krahn Inequality}~

  Let $(M, \mu)$ be a complete non-compact weighted manifold  with boundary $\delta M$ (which may be  empty).   
Moreover, we denote by  $d(x,y)$   the Riemannian distance between  $x,y\in M.$ 
  For a domain  $\Omega\subset M,$   define   
     $$\lambda_1(\Omega)=\inf_{\phi\in\mathscr C^\infty_c(\Omega)}\frac{\displaystyle\int_{\Omega}\|\nabla\phi\|^2d\mu}{\displaystyle\int_{\Omega}\phi^2d\mu}.$$ 
     Then,  $\lambda_1(\Omega)$   
     is   the first eigenvalue of $ L$ in $\Omega$ satisfying Dirichelet boundary condition on $\partial\Omega\setminus\delta M$ and Neumann boundary  condition  
     on $\delta\Omega:=\partial\Omega\cap\delta M.$
In the case when  $M=\mathbb R^m$ with $ L=\Delta$ and  Lebesgue measure $\mu,$ the classical Faber-Krahn inequality says that for any domain  $\Omega\subset\mathbb R^m$  
$$\lambda_1(\Omega)\geq C_m\mu(\Omega)^{-\frac{2}{m}}$$
for certain  constant $C_m$ such that equality is attained for balls. Faber-Krahn inequality plays a crucial  role in bounding  the heat kernel from above. 
For a general weighted manifold $(M, \mu),$  this inequality may be not  true. 
However, a  compactness argument  implies that for any geodesic ball $B(x, r),$ there are  $\alpha, b(x, r)>0$  such that for any open subset $\Omega\subset B(x,r)$  
$$\lambda_1(\Omega)\geq b(x,r)\mu(\Omega)^{-\alpha}.$$
If one finds such $\alpha$ and $b(x, r),$  then  there exists a constant $C>0$ such that 
$$p(t,x,y)\leq\frac{C\exp\left(-\frac{d(x,y)^2}{Ct}\right)}{\left(t^2b(x,\sqrt{t})b(y,\sqrt{t})\right)^{\frac{1}{2\alpha}}}$$
for all $x, y\in M$ and all $t>0$ (see  \cite{Gri3}, Theorem 5.2). 

For a  complete  weighted manifold,  
we may  consider relative Faber-Krahn inequality as a property that may  be satisfied or not, which   relates closely    to    volume doubling property and  heat kernel upper estimate  stated 
as follows. 
\begin{property}\label{ap1}  The following properties may be satisfied or not on a weighted manifold$:$  
   \begin{enumerate}
   \item[$\bullet$]     $(FK)$  Relative Faber-Krahn Inequality$:$ there exist $\alpha>0$ and $C>0$ such that for any geodesic ball $B(x, r)\subset M$ and any precompact open set 
   $\Omega\subset B(x, r),$ we have 
   $$\lambda_1(\Omega)\geq\frac{C}{r^2}\left(\frac{V(x,r)}{\mu(\Omega)}\right)^\alpha.$$
      \item[$\bullet$]     $(VD)$ Volume Doubling Property$:$ there exists $D>0$ such that  for all  $x\in M$  and all $r>0,$ we have
      $$V(x, 2r)\leq DV(x,r).$$
   \item[$\bullet$]         $(NE)$ On-Diagonal Upper Estimate$:$  there exists $C>0$ such that for all $x\in M$ and all $t>0,$ we have 
   $$p(t,x,y)\leq\frac{C}{V(x, \sqrt{t})}.$$
   \item[$\bullet$]           $(FE)$ Off-Diagonal Upper Estimate$:$  there exists $C, c>0$ such that for all $x\not=y\in M$ and all $t>0,$ we have 
   $$p(t,x,y)\leq\frac{C}{V(x, \sqrt{t})}\exp\left(-\frac{d(x,y)^2}{ct}\right).$$
\end{enumerate}
\end{property}

\begin{remark} We have  
    \begin{enumerate}
   \item[$\bullet$]  $(FK)$ holds with $\alpha=2/m$ if $M$ is a complete Riemannian manifold with  non-negative Ricci curvature $($see \cite{Gri2}, Theorem $1.4).$ 
    \item[$\bullet$] $(VD)$ implies that  there exist $\alpha, \beta, D_1, D_2>0$ such that for all $x, y\in M$ and  $0<r<R<\infty,$    we have
    $$D_1\left(\frac{R}{r}\right)^\alpha\leq\frac{V(y, R)}{V(x, r)}\leq D_2\left(\frac{d(x,y)+R}{r}\right)^\beta.$$
    \end{enumerate}
\end{remark}

  \begin{theorem}[\cite{Gri3}, Proposition 5.2]\label{equi1} Let $(M, \mu)$ be a complete weighted manifold. Then  
  $$(VD)+(NE)\Longleftrightarrow (FK)\Longleftrightarrow (VD)+(FE).$$
  \end{theorem}
  
    \begin{theorem}[\cite{Gri1}, Theorem 11.1]\label{xoxo}  Let  $(M, \mu)$ be a complete weighted manifold satisfying  $(FK).$ Then, $(M, \mu)$ is non-parabolic if and only if 
    $$\int_1^\infty\frac{dt}{V(x, \sqrt t)}<\infty.$$
      \end{theorem}

   \subsubsection{Harnack Inequality}~

The classical (elliptic) Harnack inequality concerning harmonic functions on $\mathbb R^m$ asserts  that every positive harmonic function $u$ on a ball $\mathbb B(x, r)\subset\mathbb R^m$ satisfies that 
\begin{equation}\label{zxc}
\sup_{\mathbb B(x, r/2)}u\leq H\inf_{\mathbb B(x, r/2)}u,
\end{equation}
where $H$ is a constant independent of $x, r$ and $u.$ Its best known consequence states   that any positive global harmonic function on $\mathbb R^m$ must be a constant. 
Later, J. Moser \cite{Mo} observed that (\ref{zxc})  is valid  for uniformly elliptic operators
(with measurable coefficients) in divergence form 
on $\mathbb R^m$ and that 
it implies the crucial H\"older continuity property of the local solutions
to the associated elliptic equation. For a weighted manifold, we may consider     
the above elliptic 
Harnack inequality as a property that  could  be  satisfied or not. 
 It  still seems to be an open problem  to characterize 
those weighted manifolds which  satisfy this property.
However, in  1975, Cheng-Yau \cite{C-Y}   showed 
 that for a  complete
   Riemannian manifold $(M, g)$ with   
      ${\rm{Ric}}\geq-Kg$   ($K\leq 0$),    
all positive harmonic functions $u$ on a geodesic ball $B(x, r)\subset M$
satisfies  that 
$$\|\nabla\log u\|\leq H(n)\left(K+\frac{1}{r}\right)$$
on $B(x, r/2).$ When $K=0,$ this gradient Harnack estimate can immediately imply  the validity of (\ref{zxc}). 

The parabolic version
of (\ref{zxc}) was attributed by J. Moser to J. Hadamard \cite{Ha} and B. Pini \cite{Pi}. In \cite{Mo1}, J. Moser gave the parabolic Harnack inequality for uniformly elliptic operators in divergence form on $\mathbb R^m,$ which 
 states that
any positive solution $u$ to the heat equation on a cylinder  $Q=(\tau, \tau+\alpha r^2)\times B(x, r)$ satisfies that 
 \begin{equation}\label{ttt}
   \sup_{Q_-}u\leq H\inf_{Q_+}u
   \end{equation}
for some  constant $H$ independent of  $x, r$ and $u,$ 
where 
   \begin{eqnarray*}
Q_-&=&(\tau+\alpha r^2,\tau+\beta r^2)\times B(x, \eta r), \\
  Q_+&=&(\tau+\gamma r^2,\tau+\sigma r^2)\times B(x,\eta r)
\end{eqnarray*}
for  $0<\alpha<\beta<\gamma<\sigma<\infty$ and $0<\eta<1.$ As for a complete Riemannian manifold with non-negative Ricci curvature, Li-Yau \cite{L-Y} also gave a gradient Harnack estimate for any positive solution of the heat equation on a cylinder. 
  Li-Yau's estimate implies that (\ref{ttt}) holds on complete Riemannian manifolds with non-negative Ricci curvature.

For a complete  weighted manifold, we may consider (\ref{ttt}) as a property that may  be satisfied or not, which 
  relates closely to  Poincar\'e inequality and heat kernel estimates stated as follows. 

\begin{property}\label{ap2}  The following properties  may be  satisfied  or not on a weighted manifold$:$
   \begin{enumerate}
   \item[$\bullet$]     $(PH)$  Parabolic Harnack Inequality$:$  there exists $H>0$ such that any positive solution $u(t,x)$ to the heat equation 
   $$\left(L-\frac{\partial}{\partial t}\right)u(t,x)=0$$
   on a cylinder $(\tau, \tau+4r^2)\times B(x, 2r)$  satisfies the inequality
   $$\sup_{Q_-}u\leq H\inf_{Q_+}u$$
  where 
   \begin{eqnarray*}
Q_-&=&(\tau+r^2,\tau+2r^2)\times B(x,r),  \\
  Q_+&=&(\tau+3r^2,\tau+4r^2)\times B(x,r)
\end{eqnarray*}
 for all $x\in M,$ $0<r<\infty$ and $0<\tau<\infty.$
      \item[$\bullet$]         $(PI)$ Poincar\'e Inequality$:$  there exists $C>0$ such that  for all $x\in M$ and all $r>0$ and for any $f\in\mathscr C^1(B(x,2r)),$ we have 
   $$\frac{C}{r^2}\inf_{\xi\in\mathbb R}\int_{B(x,r)}(f-\xi)^2d\mu\leq\int_{B(x,2r)}\|\nabla f\|^2d\mu.$$
   
   \item[$\bullet$]           $(HK)$  Two-Sided Heat Kernel Bounds$:$  there exist $C_1, c_1, C_2, c_2>0$ such that for all $x, y\in M$ and all $t>0,$ we have 
   $$ \ \ \ \ \ \  \ \ \ \   \    \    \frac{C_1}{V(x, \sqrt{t})}\exp\left(-\frac{d(x,y)^2}{c_1t}\right)\leq   p(t,x,y)\leq\frac{C_2}{V(x, \sqrt{t})}\exp\left(-\frac{d(x,y)^2}{c_2t}\right).$$
\end{enumerate}
\end{property}

  \begin{theorem}[\cite{Gri}, Theorem 5.1]\label{equi2}  Let $(M, \mu)$ be a complete weighted manifold.  Then 
  $$(VD)+(PI)\Longleftrightarrow(PH)\Longleftrightarrow (HK)\Longrightarrow (FK).$$
  \end{theorem}

   \subsection{Connected Sums of Riemannian Manifolds}~

Let $M_1, M_2$ be      manifolds (with  boundaries possibly) of   dimension $m.$   
  Let $\psi_j:  B_m\hookrightarrow M_j$ be an embedding for $j=1,2,$ in which $B_m$ is  the  unit ball in $\mathbb R^m.$
The \emph{connected sum} $M_1\# M_2$ of 
$M_1, M_2$ is  formed  by removing  
the ball $\psi_j(B_m)$ in  $M_j$ for all $j$ and attaching  the  punctured manifolds  $M_1\setminus \psi_1(B_m),$    
$M_2\setminus \psi_2(B_m)$
to each other via  a
 homeomorphism $\phi: \partial \psi_1(B_m)\to\partial \psi_2(B_m).$ 
That is,  
$$M_1\# M_2=(M_1\setminus \psi_1(B_m))\cup_\phi (M_2\setminus \psi_2(B_m)),$$
where $\psi_j(B_m)$ is  interior to $M_j$ and  
$\partial \psi_j(B_m)$ is bicollared in $M_j$  to guarantee that the connected sum is  a manifold (see Fig. 2).
In other words,  $M_1\# M_2$ is obtained by gluing together  the (Dirichlet) boundaries of 
$M_1\setminus \psi_1(B_m)$ and $M_2\setminus \psi_2(B_m)$ 
by means of $\psi_2\circ \psi^{-1}_1.$ 
 If $M_1, M_2$ are oriented manifolds, then $\psi_1, \psi_2$ are  
    orientation-preserving embeddings such  that 
 $M_1\# M_2$ is oriented. 
 If $M_1, M_2$ are smooth (complex) manifolds, then $\phi, \psi_1,\psi_2$  need to  be smooth  (analytic if it is  possible).  In this case, $M_1\#M_2$ is called a \emph{smooth} (\emph{complex}) \emph{connected sum} of  
 $M_1, M_2.$
                      \begin{figure}[htbp]
  \centering
  \includegraphics[width=10cm]{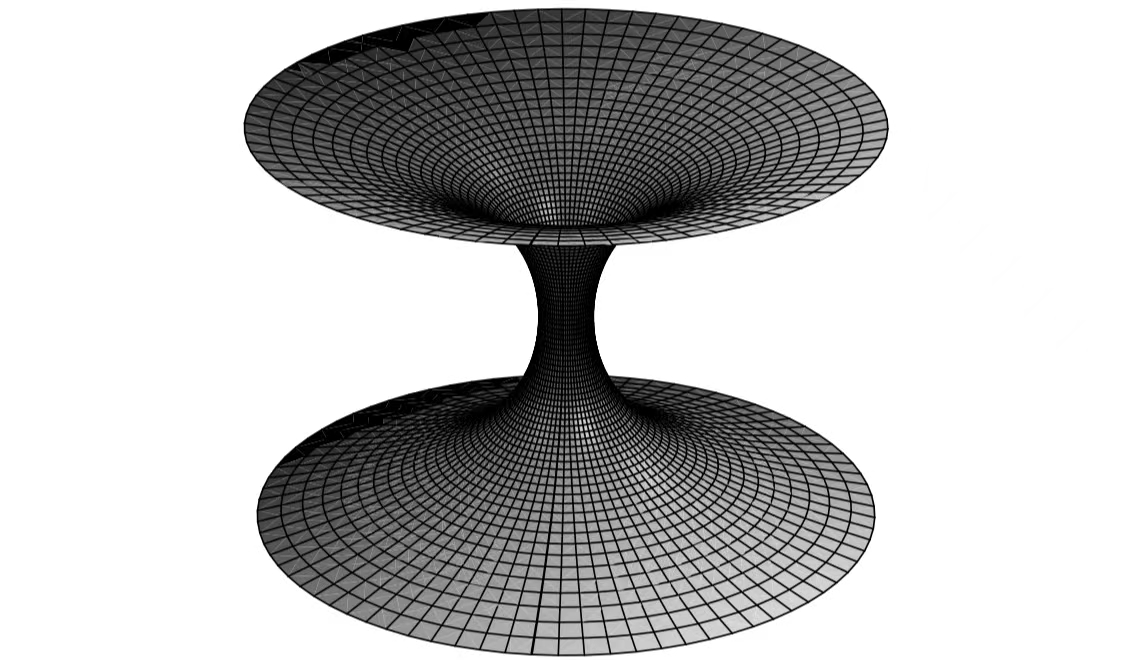}
  \caption{Catenoid}
 \end{figure}
 
  Let $M, M_1,\cdots, M_\vartheta$ be  Riemannian manifolds (with  boundaries possibly) of the same dimension. 
   Say that  $M$ is a connected sum of $M_1,\cdots, M_\vartheta$ (see Fig. 3), 
   if there  is   a compact subset $K\subset M$   (with nonempty interior)  such that $M\setminus K$ 
     is a disjoint union of $\vartheta$ open subsets $E_1,\cdots, E_\vartheta,$
     where  each   $E_j$ is  isomorphic to $M_j\setminus K_j$ for a compact subset $K_j\subset M_j.$  
 For this situation,     $M$  is called  a \emph{Riemannian connected sum} of $M_1,\cdots, M_\vartheta.$ 
        When    $M, M_j^,s$ are Hermitian  manifolds,  the  isometries need to  be analytic. 
      For   this situation,  we call $M$  a \emph{Hermitian connected sum} of $M_1,\cdots, M_\vartheta.$
             One   refers to $K$ as the \emph{central part} of $M,$ and refers  to 
  $E_1,\cdots, E_\vartheta$  as the \emph{ends} of $M.$ 
  We  also regard    
  $M_1,\cdots, M_\vartheta$ as the ends of $M$ for convenience. 
    Due to the definition, each $E_j$ can be     
  identified with 
  $M_j\setminus K_j.$
   For two weighted manifolds $(M, \mu), (M_j,\mu_j),$    
      the isometry is   understood in the sense of weighted manifolds:  it maps
      $\mu$ to   $\mu_j,$ i.e.,  
            $\mu, \mu_j$ coincide on $E_j.$

              For convenience,   the connected sum   $M=M_1\#\cdots\# M_\vartheta$ with central part $K$ and ends $E_1,\cdots, E_\vartheta$ is  also written in   the form 
     $$M=\left(K; \{M_j\}_{j=1}^\vartheta; \{E_j\}_{j=1}^\vartheta\right).$$
    \begin{figure}[htbp]
  \centering
  \includegraphics[width=11cm]{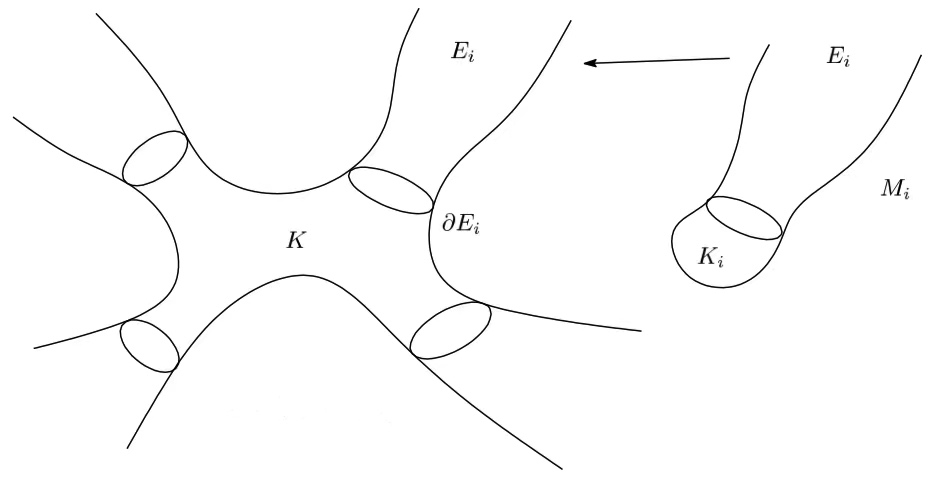}
  \caption{Manifold with ends}
 \end{figure}
 
   Let   $\partial E_j$ be   the topological 
   boundary of $E_j$ with $j=1,\cdots,\vartheta.$  
    Write   $$\partial E_j=(\partial E_j\setminus\delta M)\cup\delta E_j,$$
   where 
   $\partial E_j\setminus\delta M$   is called  the \emph{Dirichlet  boundary} of $E_j,$  and $\delta E:=\partial E_j\cap\delta M$ is called the  \emph{Neumann boundary} of $E_j.$
  It is  trivial  to examine    
$$M=K\cup(\cup_{j=1}^\vartheta E_j), \ \ \ \     \partial K=\cup_{j=1}^\vartheta (\partial E_j\setminus\delta M).$$
 \ \ \ \      Conversely, for   a non-compact  manifold $M$ (the compact case can be  dealt with similarly), if 
     there is  a compact subset $K\subset M$  with smooth boundary such that $M\setminus K$ is a disjoint union 
      of  $\vartheta$ connected open subsets $E_1,\cdots, E_\vartheta$ which are not precompact, then we say that $E_1,\cdots, E_\vartheta$ are the ends of $M.$  One considers  the closure $\overline{E}_j:=E_j\cup\partial E_j$  as a manifold with boundary $\partial E_j.$ By the definition of connected sums, we have 
      $$M=\overline{E}_1\#\cdots\#\overline{E}_\vartheta=\left(K; \{\overline{E}_j\}_{j=1}^\vartheta; \{E_j\}_{j=1}^\vartheta\right).$$
 \vskip\baselineskip

   \subsection{Almost Complex Structures of Connected Sums}~

Let $M=M_1\#\cdots\#M_\vartheta$  be a connected sum of   almost complex manifolds $M_1,\cdots, M_\vartheta$ with the same dimension $m.$
It is always  possible    concerning  the  existence of   almost complex structures for $m=2, 6,$  
because  any  orientable smooth  2-dimensional  manifold is almost complex;  
and always  possible  for  smooth orientable 6-dimensional manifolds, 
because  the only obstruction   to the existence of   almost complex structures is the \emph{third 
  Stiefel-Whitney class} $W_3$ \cite{Ma},  which is additive under the connected sum operations in this case.  
  
  Albanese-Milivojevi\'c \cite{M-A} 
       showed  the  following    existence theorem  of almost structures. 
 
\begin{theorem}[\cite{M-A}, Theorem B]
If  one of $M_j^,s$ is non-compact, then $M$ admits an almost complex structure.
\end{theorem}

However, it is possible that  the almost structures for $M$  do not exist  if all $M_j^,s$ are closed.  
For example,  we discuss  the case that $m=4k.$ 
 F. Hirzebruch (see \cite{Hir}, p. 777) proved that the Euler characteristic of $M_j$  satisfies  
$$\chi(M_j)\equiv(-1)^k\tau(M_j) \   \text{mod}   \ 4,$$
where $\tau(M_j)$
is  the  \emph{signature} of $M_j$ (see \cite{M-S}).  On the other hand, since   
        \begin{eqnarray*}
                \tau(M)&=&\sum_{j=1}^\vartheta\tau(M_j), \\
        \chi(M)&=&\sum_{j=1}^\vartheta\chi(M_j)-2(\vartheta-1),
        \end{eqnarray*}
        we have     
        $$\chi(M)\equiv (-1)^k\tau(M)-2(\vartheta-1) \   \text{mod} \   4.$$
It is therefore   
$$\chi(M)\not\equiv(-1)^k\tau(M) \   \text{mod} \   4$$
for an even number $\vartheta,$ i.e.,  $M$ does not admit   any almost complex structure when  $m=4k$ and $\vartheta=2l.$
Under  which conditions  does  $M$ possess  an almost complex structure? 
Some results of existence were obtained   in \cite{M-A, G-K, Yang1, Yang2}. 
For example,  Albanese-Milivojevi\'c \cite{M-A} showed  the following theorem.

\begin{theorem}[\cite{M-A}, Theorem C]
Assume that all $M_j^,s$ are closed, then $M$ admits an almost complex structure if  one of the following is satisfied$:$

$(a)$ $m=8k+2,$ and $\vartheta\equiv1$ {\rm{mod}} $(4k)!;$

$(b)$ $m=8k+6,$ and $\vartheta\equiv1$ {\rm{mod}} $\frac{(4k+2)!}{2}.$
\end{theorem}

   \subsection{Heat Kernel on  Connected Sums}~

Let  
     $$M=\left(K; \{M_j\}_{j=1}^\vartheta; \{E_j\}_{j=1}^\vartheta\right)$$
be a   connected sum, 
where $(M,\mu), (M_1,\mu_1),\cdots, (M_\vartheta,\mu_\vartheta)$ are complete non-compact weighted   manifolds. 
Let $p(t, x, y)$ be   the heat kernel of $(M, \mu).$  We recall 
that $B(x, r)$ is  the geodesic ball centered at $x$ with radius $r$ in $M.$
For   $x\in K\cup E_j$ with $j=1,\cdots,\vartheta,$ we put  
 $$V_j(x, r)=\mu\left(B(x,r)\cap(K\cup E_j)\right).$$  
Fix a reference point $o$ in the interior of $K.$  Set   
$$V_0(r)=\min_{1\leq j\leq\nu}V_j(r)$$
with 
$$V_j(r)=V_j(o, r), \ \ \ \     j=1,\cdots,\vartheta.$$
Again,  set for   $x\in M$ 
$$|x|=\sup_{y\in K}d(x,y).$$
Note that $|x|$ is bounded away from 0 due to the compactness of $K.$ 
Define a positive  function $H(x, t)$ on $M\times(0, \infty)$ as follows
$$H(x, t)=\min\left\{1, \  \frac{|x|^2}{V_{i_x}(|x|)}+\left(\int_{|x|^2}^t\frac{ds}{V_{i_x}(\sqrt{s})}\right)^+\right\},$$
where 
  $$ i_x=\begin{cases}
0, \ \  &   x\in K; \\
j,  \ \ &   x\in E_j \ \text{for some} \ j. 
\end{cases}$$
Evidently, $H(x,t)$ is also bounded away from 0 if $|x|$ is bounded from above. 
When $V_{i_x}(r)$ is satisfied with   
$$\frac{V_{i_x}(r_2)}{V_{i_x}(r_1)}\geq C\left(\frac{r_2}{r_1}\right)^{2+\epsilon}, \ \ \  \    r_2>r_1\geq1$$
for some $C, \epsilon>0,$ we have  
$$H(x,t)\simeq \frac{|x|^2}{V_{i_x}(|x|)}$$ 
 (see  details in the proof of Corollary 4.5,  \cite{hit}). 
   Let  $\gamma: [0,1]\to M$ be a curve with  $\gamma(0)=x$ and $\gamma(1)=y.$  Denote by 
 $\|\gamma\|$  the length of   $\gamma.$  Set 
  $$ d_\emptyset(x,y)=\begin{cases}
  \inf\big\{\|\gamma\|:    \gamma\cap K=\emptyset\big\}, &   x, y\in E_j  \ \text{for some} \ j; \\
\infty,   &  \text{otherwise}
\end{cases}$$
and
$$d_+(x,y)=\inf\big\{\|\gamma\|:   \gamma\cap K\not=\emptyset\big\}.$$
It is not hard    to verify   
\begin{equation}\label{min}
d(x, y)=\min\big\{d_\emptyset(x,y), \ d_+(x,y)\big\}.
\end{equation}
When $x,y$ belong to the same end,  we have     
$$d(x,y)\leq d_\emptyset(x,y)\leq d(x,y)+C_K,
$$
$$ \ \ \ \   |x|+|y|-2{\rm{diam}}(K)\leq d_+(x,y)\leq|x|+|y|,$$
where $C_K\geq0$ is a constant depending only on $K.$

\begin{theorem}[\cite{Gri}, Theorem 4.9]\label{upper}
Let $(M,\mu)$ be a connected sum of complete non-compact weighted manifolds $(M_1, \mu_1), \cdots, (M_\nu, \mu_\nu).$ Assume that $(M,\mu)$ is non-parabolic and  each   $(M_j, \mu_j)$ satisfies $(FK)$ in Property $\ref{ap1}.$
  Then, there exist  constants $C_2, c_2>0$ such that
\begin{eqnarray*}
p(t,x,y)&\leq& C_2\left(\frac{H(x,t)H(y,t)}{V_0(\sqrt{t})}+\frac{H(x,t)}{V_{i_y}(\sqrt{t})}+\frac{H(y,t)}{V_{i_x}(\sqrt{t})}\right)\exp\left(-\frac{d_+(x,y)^2}{c_2t}\right) \\
&&+\frac{C_2}{\sqrt{V_{i_x}(x, \sqrt{t})V_{i_y}(y, \sqrt{t})}}\exp\left(-\frac{d_\emptyset(x,y)^2}{c_2t}\right)
\end{eqnarray*}
holds  for all $x,y\in M$ and all $t>0.$
\end{theorem}

Each  term in  the upper bound of $p(t,x,y)$   in Theorem \ref{upper}  has a geometric meaning,  corresponding to  a certain  way that  a Brownian particle may move from $x$ to $y.$ To start with, the last term estimates the probability  of    getting  from  $x$ to $y$ without touching $K.$
The third term
 (similarly, the second term) estimates the probability 
   that started  at $x,$ a Brownian particle hits $K$ before  time $t$ and then  reaches  $y$  in   time of order $t.$
Finally,  the first term estimates the   probability  that     a Brownian particle     hits $K$ before   time $t,$ loops from $K$ to $K$ in time of order $t,$ and finally reaches $y$ in  time  smaller than $t.$

\begin{theorem}[\cite{Gri}, Theorem 5.10]\label{lower}
Let $(M,\mu)$ be a connected sum of complete non-compact weighted manifolds $(M_1, \mu_1), \cdots, (M_\vartheta, \mu_\vartheta).$ Assume that  each   $(M_j, \mu_j)$ is a non-parabolic manifold satisfying  $(PH)$ Property $\ref{ap2}.$
   Then,  there exist constants $C_1, c_1>0$ such that
\begin{eqnarray*}
p(t,x,y)&\geq& C_1\left(\frac{H(x,t)H(y,t)}{V_0(\sqrt{t})}+\frac{H(x,t)}{V_{i_y}(\sqrt{t})}+\frac{H(y,t)}{V_{i_x}(\sqrt{t})}\right)\exp\left(-\frac{d_+(x,y)^2}{c_1t}\right) \\
&&+\frac{C_1}{\sqrt{V_{i_x}(x, \sqrt{t})V_{i_y}(y, \sqrt{t})}}\exp\left(-\frac{d_\emptyset(x,y)^2}{c_1t}\right)
\end{eqnarray*}
holds  for all $x,y\in M$ and all $t>0.$
\end{theorem}

According to Theorem \ref{equi2}, we see that  
 $(PH)$ implies $(FK).$  
 Furthermore, using  Proposition 14.1 in  \cite{Gri1}, we also note that 
  the non-parabolicity of one  end $(M_j, \mu_j)$ implies the non-parabolicity of $(M,\mu).$
   Hence, the upper bound of  the heat kernel given  in Theorem \ref{upper} remains valid   under the conditions of Theorem \ref{lower}. 

 \section{Setups  on Complete K\"ahler Connected Sums}
 \vskip\baselineskip

\subsection{It\^o Formula}~

 Let  $M$ be a  complete Riemannian manifold with  Laplace-Beltrami operator $\Delta.$  
Let  $X_t$ be the Brownian motion  generated by  $\Delta/2$   on $M.$ 
We denote by $\mathbb P_o$ the law of $X_t$ starting from  a fixed  reference point $o\in M,$  
 and  denote by $\mathbb E_o$ the   expectation of $X_t$ with respect to $\mathbb P_o.$

 The famous  It\^o formula    (see \cite{bass, NN, Pr}) states that 
 \begin{theorem}[It\^o Formula]
 Let $\phi$ be  a   function of   $\mathscr C^2$-class on $M.$ Then  
 $$\phi(X_t)-\phi(o)=\frac{1}{2}\int_0^t\Delta\phi(X_s)ds+B\left(\int_0^t\|\nabla\phi\|^2(X_s)ds\right), \ \ \ \    \mathbb P_o-a.s.$$
  where $B_t$ is the  standard  Brownian motion on $\mathbb R.$
 \end{theorem}
Because    Brownian motion is a continuous martingale,  we have   
$$\mathbb E_o\left[B\left(\int_0^t\|\nabla\phi\|^2(X_s)ds\right)\right]=B_0=0.$$
Hence,  it yields from  It\^o formula  that 
   \begin{cor}[Dynkin Formula]\label{dynkin}
 Let $\phi$ be a  function of $\mathscr C^2$-class on $M.$ Then  
    $$\mathbb E_o\left[\phi(X_T)\right]-\phi(o)=\frac{1}{2}\mathbb E_o\left[\int_0^T\Delta\phi(X_t)dt\right]$$
    for  a finite  stopping time $T$   such  that each integrand above is  integrable. 
     \end{cor}

  Dynkin formula plays an   important role in the study of Nevanlinna theory on complex manifolds.       Let $\Omega$ be  a    bounded domain   containing $o$ with  smooth boundary $\partial \Omega$ in $M.$ 
Denote by   $g_\Omega(x,y)$    the Green function of $\Delta/2$ for  $\Omega$ with  a pole at $o$ satisfying Dirichlet boundary condition, 
         and    by 
        $\pi_{\partial \Omega}$    the harmonic measure on $\partial \Omega$ with respect to $o.$ Note that 
      $$d\pi_{\partial\Omega}=\frac{1}{2}\frac{\partial g_\Omega(o,x)}{\partial{\vec{\nu}}}d\sigma_{\partial\Omega},$$
  where  $\partial/\partial \vec\nu$ is  the inward  normal derivative on $\partial \Omega(r),$  $d\sigma_{\partial\Omega}$ is the Riemannian area element  of 
$\partial \Omega(r).$
    
Historically, the first theorem    indicating   the  connection 
 between Brownian motion and potential theory was obtained   by S. Kakutani \cite{Ka} in 1944.  
 The  further   relation  was  investigated   
  by J. Doob \cite{Do1,Do2}, G.  Hunt \cite{Hu},  
 A. Knapp \cite{Kn}, N. Privault \cite{Pr} and Port-Stone \cite{PS}, etc. 
In the following, we  introduce two  classical  formulas   relating  closely  to  Nevanlinna theory. 

Let $\phi$ be a  $\mathscr C^2$-function   outside  a  possible polar set of singularities at most on $M.$ Set   the  first existing  time for $\Omega$
$$\tau_{\Omega}=\inf\big\{t>0:  X_t\not\in \Omega\big\}.$$
  The co-area formula reads  (see  \cite{at1,at2, bass})
 \begin{equation}\label{coa}
  \int_{\Omega}g_{\Omega}(o,x)\Delta\phi dv= \mathbb{E}_o\left[\int_0^{\tau_{\Omega}}\Delta\phi(X_t)dt\right],
 \end{equation}
 where $dv$ is the Riemannian volume element of $M.$ 
The mean-value integral  has the following  probabilistic expression  (see \cite{bass}, Proposition 2.8;  and \cite{at1,at2}) 
\begin{equation}\label{mean}
 \int_{\partial \Omega}\phi d\pi_{\partial \Omega}= \mathbb{E}_o\left[\phi(X_{\tau_{\Omega}})\right]. 
\end{equation}
Set  $T=\tau_{\Omega}$  and apply     (\ref{coa}) with (\ref{mean}) to Corollary \ref{dynkin},   the Dynkin formula also holds  in   the  viewpoint  of distributions since the possible set of singularities of $\phi$ is polar. Hence, we have: 
     \begin{cor}[Jensen-Dynkin Formula]\label{J-D}
 Let $\phi$ be a  $\mathscr C^2$-function   outside  a  possible polar set of singularities at most on $M.$  Assume that $\phi(o)\not=\infty.$  Then  
    $$\int_{\partial \Omega}\phi d\pi_{\partial \Omega}-\phi(o)=\frac{1}{2}\int_{\Omega}g_{\Omega}(o,x)\Delta\phi dv.$$
 \end{cor}  

When $M=\mathbb C^m,$  the Green function and harmonic measure for  balls  have explicit expressions. 
By a direct computation, we will see that Green-Jensen formula (see \cite{No, ru}) follows from Jensen-Dynkin formula.  
 
 \subsection{Complete K\"ahler Connected Sums}~

 We first  give  the notion of K\"ahler connected sums as follows.
 
 \begin{defi}\label{defcc}
Let $M$ be a   K\"ahler  manifold. We say that $M$ is  a  K\"ahler connected sum with $\vartheta$ ends,  if  
    there exist  $\vartheta$   K\"ahler manifolds $M_1,\cdots, M_\vartheta$ 
 such that 
      $M=M_1\#\cdots\#M_\vartheta.$ In addition, we say that $M$ is complete if all $M_j^,s$ are complete. 
      \end{defi}
       
        Let 
     $$M=\left(K; \{M_j\}_{j=1}^\vartheta; \{E_j\}_{j=1}^\vartheta\right)$$
 be  a complete K\"ahler connected sum with $\vartheta$ non-parabolic ends (by which we mean that all $M_j^,s$ are non-parabolic),  where  all $M_j^{,}s$ are  assumed to be non-compact.  
As noted   in Section 2.5,   $M$ is  also non-parabolic.

 Let $\Delta$ denote    the  Laplace-Beltrami operator  on $M.$ 
We  consider the heat kernel $p(t, x, y)$ of $M,$  which is the minimal positive   fundamental solution to the  heat equation
$$\left(\Delta-\frac{\partial}{\partial t}\right)u(t,x)=0.$$
\ \ \ \   We  first  review  some notations.  Fix a  reference point $o$ in the interior of $K.$  Denote by $B(r)$   the geodesic ball  centered at $o$ with radius $r$ in $M,$ and by $d(x,y)$ the Riemannian distance between $x,y\in M.$
     Set 
\begin{equation}\label{ref0}
V_{\rm min}(r)=\min_{1\leq j\leq\vartheta}V_j(r), \ \ \ \   V_{\rm max}(r)=\max_{1\leq j\leq\vartheta}V_j(r)
\end{equation}
with 
$$V_j(r)={\rm{Vol}}\left(B(r)\cap (K\cup E_j)\right), \ \ \ \  j=1,\cdots,\vartheta.$$
Put  
$$\rho(x)=d(o,x), \ \ \ \    |x|=\sup_{y\in K}d(x,y).$$
 Again, define the function $H(x, t)$ on $M\times(0, \infty)$ as follows
 \begin{equation}\label{H}
H(x, t)=\min\left\{1, \  \frac{|x|^2}{V_{i_x}(|x|)}+\left(\int_{|x|^2}^t\frac{ds}{V_{i_x}(\sqrt{s})}\right)^+\right\},
\end{equation}
where 
  $$ i_x=\begin{cases}
j, \ &  x\in E_j \ \text{for  some} \ j;  \\
0, \ &  x\in K.
\end{cases}$$

   We  give    two-sided bounds   of $p(t, o, x)$  using   Theorems \ref{upper} and \ref{lower}.
Since $o\in K,$   it follows from (\ref{min}) and  definitions of $d_\emptyset(x,y)$ and $d_+(x,y)$ that  
\begin{equation}\label{ref1}
d_\emptyset(o,x)=\infty, \ \ \ \    d_+(o,x)=\rho(x). 
\end{equation}
Notice  that $K\subset M$ is a compact subset with nonempty  interior,  we obtain  $0<|o|<\infty.$ By $i_o=0$ and (\ref{H}),  it is not difficult  to derive  
\begin{equation}\label{ref2}
0<\gamma_0:=\min\left\{1,\  \frac{|o|^2}{V_{\rm max}(|o|)}\right\}\leq H(o,t)\leq1, \ \ \ \  0<H(x, t)\leq1.
\end{equation}
Set
$$W(x,t)=\frac{H(o,t)H(x,t)}{V_{\min}(\sqrt{t})}+\frac{H(o,t)}{V_{i_x}(\sqrt{t})}+\frac{H(x,t)}{V_{i_o}(\sqrt{t})}.$$
Then,  (\ref{ref0}) and (\ref{ref2}) yield  that 
\begin{equation}\label{ref3}
\frac{\gamma_0}{V_{\rm max}(\sqrt{t})}\leq W(x,t)\leq\frac{3}{V_{\rm min}(\sqrt{t})}. 
\end{equation}
Note that $(PH)$ implies $(FH)$ and the non-parabolicity of an end $M_j$ implies the non-parabolicity of $M.$ Combining (\ref{ref1}) with (\ref{ref3}) and utilizing  Theorems \ref{upper} and \ref{lower}, we obtain: 
\begin{cor}\label{estimate}
Let $M=(K; \{M_j\}_{j=1}^\vartheta; \{E_j\}_{j=1}^\vartheta)$ be a  complete K\"ahler connected sum with  $\vartheta$ non-parabolic  ends,  where all  $M_j^,s$ are non-compact manifolds   
 satisfying  $(PH)$ in Property $\ref{ap2}.$  
 Then,  there exist constants $A, B, a, b>0$ such that
\begin{equation*}
\frac{A}{V_{\rm max}(\sqrt{t})}\exp\left(-\frac{\rho(x)^2}{at}\right) \leq p(t,o,x) \leq \frac{B}{V_{\rm min}(\sqrt{t})}\exp\left(-\frac{\rho(x)^2}{bt}\right)
\end{equation*}
holds  for all $x\in M$ and all $t>0.$ 
\end{cor}

 \subsection{Nevanlinna's Functions via Heat Kernel}~

Let
  $$M=M_1\#\cdots\#M_\vartheta$$
 be  a complete K\"ahler connected sum with $\vartheta$ non-parabolic ends, in which  all  $M_j^{,}s$ are   non-compact manifolds  satisfying   $(PH)$  in Property \ref{ap2}, of complex dimension $m.$ 
  The  K\"ahler form of $M$  associated to  K\"ahler metric $g=(g_{i\bar j})$      is defined  by 
 $$\alpha=\frac{\sqrt{-1}}{\pi}\sum_{i,j=1}^mg_{i\bar j}dz_i\wedge d\bar z_{j}$$
 in  holomorphic local  coordinates $z_1,\cdots,z_m.$ 
We   adopt    the same notations as defined in Section 3.2.
  By Corollary \ref{estimate}, there exist  constants $A, B, a, b>0$ such that
\begin{equation}\label{est}
\frac{A}{V_{\rm max}(\sqrt{t})}\exp\left(-\frac{\rho(x)^2}{at}\right) \leq p(t,o,x) \leq \frac{B}{V_{\rm min}(\sqrt{t})}\exp\left(-\frac{\rho(x)^2}{bt}\right)
\end{equation}
holds  for all $x\in M$ and all $t>0.$  By the  non-parabolicity of  $M,$   the  following infinite integral 
 $$G(o,x):=2\int_0^\infty p(t,o,x)dt$$
   converges  for $x\not=o,$ which thus  defines   the minimal positive Green function of $\Delta/2$ for $M$ with a pole at $o.$  The   non-parabolicity of all $M_j^,s$  imply  that 
$$\int_0^\infty\frac{e^{-\frac{\rho(x)^2}{at}}}{V_{\rm max}(\sqrt{t})}dt, \ \ \ \   \int_0^\infty\frac{e^{-\frac{\rho(x)^2}{bt}}}{V_{\rm min}(\sqrt{t})}dt$$
 are convergent   for $x\not=o$  due to Theorem \ref{xoxo}.  Moreover,  each of them  tends to $0$ as $x\to\infty$ (where $``x\to\infty"$ means that $\rho(x)\to\infty$) and tends to $\infty$ as $x\to  o.$
  It follows from Corollary \ref{estimate} that 
 \begin{equation}\label{green}
2A\int_0^\infty\frac{e^{-\frac{\rho(x)^2}{at}}}{V_{\rm max}(\sqrt{t})}dt \leq G(o,x) \leq 2B\int_0^\infty\frac{e^{-\frac{\rho(x)^2}{bt}}}{V_{\rm min}(\sqrt{t})}dt
\end{equation}
holds  for all $x\in M.$ 
 
 To define Nevanlinna's functions,  we  need      to construct a family $\{\Delta(r)\}_{r>0}$ of   precompact  domains  containing
  $o,$ being able to    exhaust $M.$ One workable way is to  utilize  the heat kernel $p(t,o,x),$   because  it is  effective    for   
  estimating  the  Green function   for $\Delta(r)$ (with a pole at $o$ satisfying Dirichelet boundary condition). 
 For $r>0,$ define
  $$\Delta(r)=\left\{x\in M: \ G(o,x)> 2A\int_0^\infty\frac{e^{-\frac{r^2}{at}}}{V_{\rm max}(\sqrt{t})}dt\right\}.$$
  According to  
  $$\int_0^\infty\frac{e^{-\frac{r^2}{At}}}{V_{\rm max}(\sqrt{t})}dt\searrow0, \ \ \ \ r\to\infty$$
and   
$$\lim_{x\to o}G(o,x)=\infty, \ \ \ \     \lim_{x\to\infty}G(o,x)=0,$$
  we deduce  that  $\Delta(r)$ is a precompact  domain containing $o$ in $M$ satisfying  
 $$ \ \ \   \lim_{r\to0}\Delta(r)\to \emptyset, \ \ \ \   \lim_{r\to\infty}\Delta(r)=M.$$
 In fact,  the family 
     $\{\Delta(r)\}_{r>0}$ exhausts $M.$ That is,  for any   sequence $\{r_n\}_{n=1}^\infty$ such that    $0<r_1<r_2<\cdots\to \infty,$ we have  
 $$\bigcup_{n=1}^\infty\Delta(r_n)=M$$
 and
$$\emptyset\not=\Delta(r_1)\subset\overline{\Delta(r_1)}\subset\Delta(r_2)\subset\overline{\Delta(r_2)}\subset\cdots.$$  
  So, the boundary of $\Delta(r)$ can be formulated as   
  $$\partial\Delta(r)=\left\{x\in M: \ G(o,x)=2A\int_0^\infty\frac{e^{-\frac{r^2}{at}}}{V_{\rm max}(\sqrt{t})}dt\right\}. \ $$
   By Sard's theorem,  each connected component of  $\partial\Delta(r)$ is a submanifold of $M$ for almost all $r>0.$
    Define 
 $$g_r(o,x)=G(o,x)-2A\int_0^\infty\frac{e^{-\frac{r^2}{at}}}{V_{\rm max}(\sqrt{t})}dt.$$
 Then, we see that  $g_r(o,x)$ defines   the Green function of $\Delta/2$ for $\Delta(r)$ with a pole at $o$ satisfying Dirichelet boundary condition. 
 Let  $\pi_r$  be the harmonic measure  on $\partial\Delta(r)$ with respect to $o,$ i.e., 
  $$d\pi_r=\frac{1}{2}\frac{\partial g_r(o,x)}{\partial{\vec{\nu}}}d\sigma_r,$$
  where  $\partial/\partial \vec\nu$ is the inward  normal derivative on $\partial \Delta(r),$ $d\sigma_{r}$ is the Riemannian area element  of 
$\partial \Delta(r).$

   Next, we  define  the Nevanlinna's functions. 
   Let $X$ be a complex projective manifold, over which we can  put  
     a Hermitian positive  line bundle $(L, h)$
    with  Chern form  $c_1(L,h):=-dd^c\log h>0,$ where 
$$d=\partial+\overline{\partial}, \ \ \ \    d^c=\frac{\sqrt{-1}}{4\pi}(\overline{\partial}-\partial)$$
so that 
$$dd^c=\frac{\sqrt{-1}}{2\pi}\partial\overline{\partial}.$$
\ \ \ \  Let $f: M\to X$ be a meromorphic mapping. The \emph{characteristic function} of $f$ with respect to $L$ is defined by 
    \begin{eqnarray*}
 T_f(r, L) &=& \frac{\pi^m}{(m-1)!}\int_{\Delta(r)}g_r(o,x)f^*c_1(L,h)\wedge\alpha^{m-1}  \\
 &=&    -\frac{1}{4}\int_{\Delta(r)}g_r(o,x)\Delta\log(h\circ f)dv, 
     \end{eqnarray*}
   where    $dv$ is the  Riemannian volume element  of $M.$ 

  Let $s_D$ be the  canonical section  associated to $D\in|L|.$  That is, $s_D$ is the holomorphic section of $L$ over $X$ with zero divisor $D.$ 
The \emph{proximity function} of $f$ with respect to $D$ is defined  by
 $$m_f(r,D)=\int_{\partial\Delta(r)}\log\frac{1}{\|s_D\circ f\|}d\pi_r.$$
 Meanwhile,   we define the  \emph{counting function}  and   \emph{simple counting function} of $f$ with respect to $D$ respectively  by 
   \begin{eqnarray*}
N_f(r,D)&=& \frac{\pi^m}{(m-1)!}\int_{f^*D\cap\Delta(r)}g_r(o,x)\alpha^{m-1}, \\
\overline{N}_f(r, D)&=& \frac{\pi^m}{(m-1)!}\int_{f^{-1}(D)\cap\Delta(r)}g_r(o,x)\alpha^{m-1}. 
 \end{eqnarray*}
 The characteristic function of  Ricci form  $\mathscr R=-dd^c\log\det(g_{i\bar j})$  is defined by
     \begin{eqnarray*}
 T(r,\mathscr R)&=& \frac{\pi^m}{(m-1)!}\int_{\Delta(r)}g_r(o,x)\mathscr R\wedge\alpha^{m-1} \\
 &=&  -\frac{1}{4}\int_{\Delta(r)}g_r(o,x)\Delta\log\det(g_{i\bar j})dv. 
     \end{eqnarray*}
 \ \ \ \    When $M=\mathbb C^m$ 
    \begin{equation}\label{t4}
 p(t, \textbf{0},z)=\frac{1}{(4\pi t)^{m}}e^{-\frac{\|z\|^2}{4t}}.
     \end{equation}
  We consider the non-parabolic case (i.e., $m\geq2$). By definition   
 \begin{equation*}\label{f6}
 G(\textbf{0}, z)=2\int_0^\infty\frac{e^{-\frac{\|z\|^2}{4t}}}{(4\pi t)^{m}}dt,
  \end{equation*}          
  which gives  that 
        \begin{eqnarray}\label{g4}
       \Delta(r)&=& \left\{x\in \mathbb C^m: \ G(\textbf{0},z)> 2\int_0^\infty \frac{e^{-\frac{r^2}{4t}}}{(4\pi t)^{m}}dt\right\}  \\
       &=& \big\{z\in\mathbb C^m:   \|z\|<r\big\}  
     := \mathbb B(r) \nonumber
          \end{eqnarray}
and   
        \begin{eqnarray*}\label{g5}
 g_r(\textbf{0}, z) 
 &=& 2\int_0^\infty\frac{e^{-\frac{\|z\|^2}{4t}}}{(4\pi t)^{m}}dt-2\int_0^\infty \frac{e^{-\frac{r^2}{4t}}}{(4\pi t)^{m}}dt \\
 &=& \frac{\|z\|^{2-2m}-r^{2-2m}}{(m-1)\omega_{2m-1}}, 
          \end{eqnarray*}
where $\omega_{2m-1}$ is the area of  the unit sphere in $\mathbb C^m.$ In further, we derive    
 $$d\pi_r(z)=d^c\log\|z\|^2\wedge \left(dd^c\log\|z\|^2\right)^{m-1}.$$
By integration-by-parts formula, we will see  that our Nevanlinna's functions agree with the classical ones. 
Notice  the connection  between potential theory and Brownian motion, 
  Nevanlinna's functions have  alternative probabilistic expressions  as  follows. 

Let $\mathbb P_o$ be the law of Brownian motion  $X_t$ on $M$ generated by $\Delta/2$ starting  from $o.$ The  corresponding  expectation of $X_t$ is denoted by  $\mathbb E_o.$   Set the first existing time for $\Delta(r)$  
 $$\tau_r=\inf\big\{t>0:  X_t\not\in \Delta(r)\big\}.$$
By co-area formula (see (\ref{coa})), we obtain   
     \begin{eqnarray*}
T_f(r,L)& &=-\frac{1}{4}\mathbb E_o\left[\int_0^{\tau_r}\Delta\log(h\circ f(X_t))dt\right], \\
T(r,\mathscr R)& &=-\frac{1}{4}\mathbb E_o\left[\int_0^{\tau_r}\Delta\log\det(g_{i\bar j}(X_t))dt\right].
     \end{eqnarray*}
Also,  (\ref{mean}) yields  that 
$$  m_f(r,D) =\mathbb E_o\left[\log\frac{1}{\|s_D\circ f(X_{\tau_r})\|}\right].$$
 Moreover, we  have (see \cite{at1, at2, carne, Dong1})
 $$ N_f(r,D)=\lim_{\lambda\rightarrow\infty}\lambda\mathbb P_o\left(\sup_{0\leq t\leq\tau_r}\log\frac{1}{\|s_D\circ f(X_t)\|}>\lambda\right).$$
However, there exists  no  such an analytic  expression for $\overline{N}_f(r, D)$ in general.  
Using  Poincar\'e-Lelong formula (see \cite{Dem, No, ru}) and Jensen-Dynkin formula (see Corollary \ref{J-D}), it is immediate to  derive    the following first main theorem. 
\begin{theorem}[\cite{Dong1}]  Assume   that $f(o)\not\in{\rm{Supp}}D.$ Then
$$T_f(r,L)+\log\frac{1}{\|s_D\circ f(o)\|}=m_f(r,D)+N_f(r,D).$$
\end{theorem}
\vskip\baselineskip

\subsection{Ahlfors-Shimizu's Form of $T_f(r,L)$}~

In order to  better   understand  the   characteristic function $T_f(r,L),$  
we need to  describe  how similar  $T_f(r,L)$  is to the classical characteristic function.  
We shall prove  that $T_f(r,L)$ has an alternative form very like 
 Ahlfors-Shimizu's  characteristic function. 

\begin{lemma}\label{grest}   For  $0<t\leq r,$ we have 
 $$g_r(o,x)=2A\int_0^\infty\frac{e^{-\frac{t^2}{as}}-e^{-\frac{r^2}{as}}}{V_{\max}(\sqrt{s})}ds$$
   for  all $x\in\partial\Delta(t).$
 \end{lemma}
 \begin{proof}
Bt the  definition   of $\partial\Delta(t),$ we have     
 $$g_t(o,x)=G(o,x)-2A\int_0^\infty\frac{e^{-\frac{t^2}{as}}}{V_{\rm max}(\sqrt{s})}ds=0$$
 for   $x\in\partial\Delta(t).$
It is therefore   
    \begin{eqnarray*}
g_r(o,x)&=& G(o,x)-2A\int_0^\infty\frac{e^{-\frac{r^2}{at}}}{V_{\rm max}(\sqrt{t})}dt \\
&=& G(o,x)-2A\int_0^\infty\frac{e^{-\frac{t^2}{as}}}{V_{\rm max}(\sqrt{s})}ds+2A\int_0^\infty\frac{e^{-\frac{t^2}{as}}}{V_{\rm max}(\sqrt{s})}ds \\
&& -2A\int_0^\infty\frac{e^{-\frac{r^2}{as}}}{V_{\rm max}(\sqrt{s})}ds\\
&=& 2A\int_0^\infty\frac{e^{-\frac{t^2}{as}}-e^{-\frac{r^2}{as}}}{V_{\max}(\sqrt{s})}ds
   \end{eqnarray*}
    for  $x\in\partial\Delta(t).$
 \end{proof}
 
   Lemma \ref{grest} is  useful    for the establishment of   calculus lemma  in  Section 4.
 Here, we give  its another application. 
 
 \begin{theorem}\label{f5} $T_f(r, L)$ has the  Ahlfors-Shimizu's form
         \begin{eqnarray*}   
  T_f(r, L)  & =& \frac{4A\pi^m}{(m-1)!a}\int_0^rt\int_0^\infty\frac{e^{-\frac{t^2}{as}}}{V_{\max}(\sqrt{s})}\frac{ds}{s}dt\int_{\Delta(t)}f^*c_1(L,h)\wedge\alpha^{m-1}. 
          \end{eqnarray*}   
\end{theorem}

\begin{proof}     
By integration-by-parts formula and  Lemma \ref{grest}, we have  
        \begin{eqnarray*}   
&& \frac{4A\pi^m}{(m-1)!a}\int_0^rt\int_0^\infty\frac{e^{-\frac{t^2}{as}}}{V_{\max}(\sqrt{s})}\frac{ds}{s}dt\int_{\Delta(t)} f^*c_1(L,h)\wedge\alpha^{m-1} \\
&=& -\frac{A}{a}\int_0^rt\int_0^\infty\frac{e^{-\frac{t^2}{as}}}{V_{\max}(\sqrt{s})}\frac{ds}{s}dt\int_{\Delta(t)} \Delta\log(h\circ f)dv \\
&=& \frac{A}{2}\int_0^rd\int_0^\infty\frac{e^{-\frac{t^2}{as}}-e^{-\frac{r^2}{as}}}{V_{\max}(\sqrt{s})}ds
\int_{\Delta(t)} \Delta\log(h\circ f)dv \\
&=& \frac{A}{2}\left[\int_0^\infty\frac{e^{-\frac{t^2}{as}}-e^{-\frac{r^2}{as}}}{V_{\max}(\sqrt{s})}ds
\int_{\Delta(t)} \Delta\log(h\circ f)dv\right]_{t=0}^r \\
&& -\frac{A}{2}\int_0^r\int_0^\infty\frac{e^{-\frac{t^2}{as}}-e^{-\frac{r^2}{as}}}{V_{\max}(\sqrt{s})}ds
d\int_{\Delta(t)} \Delta\log(h\circ f)dv \\
&=& -\frac{A}{2}\int_0^rdt\int_0^\infty\frac{e^{-\frac{t^2}{as}}-e^{-\frac{r^2}{as}}}{V_{\max}(\sqrt{s})}ds
\int_{\partial\Delta(t)} \Delta\log(h\circ f)d\sigma_t \\
&=&  -\frac{A}{2}\int_0^rdt
\int_{\partial\Delta(t)}\int_0^\infty\frac{e^{-\frac{t^2}{as}}-e^{-\frac{r^2}{as}}}{V_{\max}(\sqrt{s})}ds\Delta\log(h\circ f)d\sigma_t. \\
&=&  -\frac{1}{4}\int_0^rdt\int_{\partial\Delta(t)}g_r(o,x)\Delta\log(h\circ f)d\sigma_t \\
 &=&    -\frac{1}{4}\int_{\Delta(r)}g_r(o,x)\Delta\log(h\circ f)dv \\
 &=& T_f(r,L).
        \end{eqnarray*}
This completes the proof. 
\end{proof}

 Locally, write $s_D=\tilde s_De,$ where $e$ is a  holomorphic local frame of $L.$ Due to   Poincar\'e-Lelong formula, we have 
$$\left[D\right]=dd^c\left[\log|\tilde s_D\circ f|^2\right]$$
in the sense of currents.  The similar  argument as in the proof of Lemma \ref{grest}   leads  to     
      \begin{eqnarray*}
 N_f(r,D) 
&=& \frac{4A\pi^m}{(m-1)!a}\int_0^rt\int_0^\infty\frac{e^{-\frac{t^2}{as}}}{V_{\max}(\sqrt{s})}\frac{ds}{s}dt\int_{f^*D\cap\Delta(t)}\alpha^{m-1},  \\
\overline{N}_f(r, D)&=& \frac{4A\pi^m}{(m-1)!a}\int_0^rt\int_0^\infty\frac{e^{-\frac{t^2}{as}}}{V_{\max}(\sqrt{s})}\frac{ds}{s}dt\int_{f^{-1}(D)\cap\Delta(t)}\alpha^{m-1}. 
 \end{eqnarray*} 
 \ \ \ \    To see it  more clearly,  we consider  $M=\mathbb C^m$  ($m\geq2$) as an example.  
 Due to  (\ref{t4}) and (\ref{g4}), we get  
 $$V_{\max}(\sqrt t)=\frac{\pi^m}{m!}t^m, \ \ \  \    a=4, \ \ \ \   A=\frac{1}{4^{m}m!},$$
which yields that  
      \begin{eqnarray*}
 \frac{4A\pi^m}{(m-1)!a}\int_0^\infty\frac{e^{-\frac{t^2}{as}}}{V_{\max}(\sqrt s)}\frac{ds}{s} 
=
 \frac{1}{4^m(m-1)!} \int_0^\infty e^{-\frac{t^2}{4s}}\frac{ds}{s^{m+1}} 
 =\frac{1}{t^{2m}}.
      \end{eqnarray*}      
By  this  with Theorem \ref{f5}, we are led to     
         \begin{eqnarray*}   
  T_f(r, L)   &=&\int_0^r\frac{dt}{t^{2m-1}}\int_{\mathbb B(t)}f^*c_1(L,h)\wedge\alpha^{m-1}, 
          \end{eqnarray*}   
which is  just the  Ahlfors-Shimizu's  characteristic function. 

\section{Two Key Lemmas}
 
The main purpose  in this section is to establish two key lemmas:  calculus lemma and logarithmic derivative lemma.
 Let. 
  $$M=M_1\#\cdots\#M_\vartheta$$
 be  a complete K\"ahler connected sum with $\vartheta$ non-parabolic ends, where  all  $M_j^{,}s$ are non-compact manifolds 
 satisfying   $(PH)$  in Property $\ref{ap2},$ with Ricci curvature bounded from below by a constant.

  \subsection{Calculus Lemma}~
  
   Let $\nabla$ denote the gradient operator on  Riemannian manifolds.  Cheng-Yau \cite{C-Y} proved the following theorem.

 \begin{lemma}\label{CY}
 Let $N$ be a complete Riemannian manifold.  Let   $B(x_0, r)$ be a geodesic ball centered at $x_0\in N$ with radius $r.$ 
 Then, there exists  a constant $c_N>0$ depending only on the dimension of $N$ such that 
 $$\frac{\|\nabla u(x)\|}{u(x)}\leq \frac{c_N r^2}{r^2-d(x_0, x)^2}\left(|R(r)|+\frac{1}{d(x_0, x)}\right)$$
 holds for any  non-negative  harmonic function $u$ on $B(x_0, r),$ 
 where $d(x_0, x)$ is the Riemannian distance between $x_0$ and $x,$ and $R(r)$ is the lower bound of  Ricci curvature of  $B(x_0, r).$
 \end{lemma}
 
Lemma \ref{CY} yields that (see \cite{L-T1}, Lemma 3.1   also)
 
    \begin{cor}\label{CY1}
 Let $N$ be a complete Riemannian manifold with non-negative Ricci curvature outside a compact subset  $K.$  
 Let $B(x_0, r_j)\supset K$ $(j=1,2)$ be the geodesic balls centered at $x_0\in N$ with radius $r_j$ satisfying $r_1<r_2.$ 
   Then, there exists  a constant $c_N>0$ depending only on the dimension of $N$ such that 
 $$\frac{\|\nabla u(x)\|}{u(x)}\leq \frac{c_N r_2^2}{r_2^2-d(x_0, x)^2}\frac{1}{d(x_0, x)}$$
 holds for any  non-negative  harmonic function $u$ on $B(x_0, r_2)\setminus B(x_0, r_1),$ 
 where $d(x_0, x)$ is the Riemannian distance between $x_0$ and $x.$ 
  \end{cor}
   
 \begin{theorem}\label{hh} There exists a constant  $c_1>0$   such that 
 $$\|\nabla g_r(o,x)\|\leq c_1\left(|\kappa|+r^{-1}\right)\int_0^\infty\frac{e^{-\frac{r^2}{bt}}}{V_{\rm min}(\sqrt{t})}dt$$
 holds  for all $x\in\partial\Delta(r)$ satisfying  $K\subset\Delta(r),$ where $K$ is the central part of $M$ and  $\kappa$ is defined by $(\ref{min1}).$  
\end{theorem}
\begin{proof}  
Since    $M_j$  has the Ricci curvature bounded from below by a constant $\kappa_j$ for $j=1,\cdots,\vartheta,$ 
 we infer   that   the Ricci curvature of $M$ is bounded from below by a constant,  denoted by $R.$
 It follows  from Lemma \ref{CY} and  (\ref{green})  that   (see Remark 5 in \cite{TS} also)
  \begin{eqnarray*}
\|\nabla G(o, x)\|&\leq& c_M\left(|R|+\rho(x)^{-1}\right)G(o,x) \\
&\leq& c_M\left(|R|+\rho(x)^{-1}\right)\int_0^\infty\frac{e^{-\frac{\rho(x)^2}{bt}}}{V_{\rm min}(\sqrt{t})}dt
  \end{eqnarray*}
 for some large  constant $c_M>0$ which depends  only on  the dimension  of $M.$   By  (\ref{green})   again 
 $$\int_0^\infty\frac{e^{-\frac{\rho(x)^2}{at}}}{V_{\rm max}(\sqrt{t})}dt \leq \int_0^\infty\frac{e^{-\frac{r^2}{at}}}{V_{\rm max}(\sqrt{t})}dt$$
  for  $x\in\partial\Delta(r),$  which  gives    
 $\rho(x)\geq r$
  for  $x\in\partial\Delta(r).$
Set $c_0=2c_MB,$ then we obtain    
      \begin{eqnarray*} 
\|\nabla g_r(o, x)\|\big|_{\partial\Delta(r)}&=&\|\nabla G(o, x)\|\big|_{\partial\Delta(r)} \\
&\leq& c_0\left(|R|+r^{-1}\right)\int_0^\infty\frac{e^{-\frac{r^2}{bt}}}{V_{\rm min}(\sqrt{t})}dt.
      \end{eqnarray*}
   Hence,   for  $\kappa\not=0$  
            \begin{eqnarray*} 
\|\nabla g_r(o, x)\|\big|_{\partial\Delta(r)}
&\leq& c_1\left(|\kappa|+r^{-1}\right)\int_0^\infty\frac{e^{-\frac{r^2}{bt}}}{V_{\rm min}(\sqrt{t})}dt,
      \end{eqnarray*}
where $c_1=\max\{c_0, |c_0R/\kappa|\}.$ If $\kappa=0,$ then all $M_j^,s$ have non-negative Ricci curvature. So,  using  Corollary \ref{CY1} and (\ref{green}), 
we can also obtain
          \begin{eqnarray*} 
\|\nabla g_r(o, x)\|\big|_{\partial\Delta(r)}
&\leq& \frac{c_1}{r}\int_0^\infty\frac{e^{-\frac{r^2}{bt}}}{V_{\rm min}(\sqrt{t})}dt
      \end{eqnarray*}
for some large constant $c_1>0$   if    $K\subset \Delta(r).$  
       This completes the proof. 
  \end{proof}

   By
   $$\frac{\partial g_r(o,x)}{\partial{\vec{\nu}}}=\|\nabla g_r(o, x)\|$$
  for  $x\in\partial\Delta(r),$ 
we  obtain   an estimate for upper bounds of  $d\pi_r$  as follows.

 \begin{cor}\label{pop} There exists a constant  $c_2>0$   such that 
 $$d\pi_r\leq c_2\left(|\kappa|+r^{-1}\right)\int_0^\infty\frac{e^{-\frac{r^2}{bt}}}{V_{\rm min}(\sqrt{t})}dt\cdot d\sigma_r$$
 holds  for all $x\in\partial\Delta(r)$ satisfying  $K\subset\Delta(r),$ where $K$ is the central part of $M,$  $d\sigma_r$ is the Riemannian area element of $\partial\Delta(r)$
 and  $\kappa$ is defined by $(\ref{min1}).$  
 \end{cor}
    
 \begin{lemma}[Borel's Lemma]\label{} Let $u\geq0$ be a non-decreasing  function  on $(r_0, \infty)$ with $r_0\geq0.$ Then for any $\delta>0,$ there exists a subset $E_\delta\subset(r_0,\infty)$
 of finite Lebesgue measure such that  
 $$u'(r)\leq u(r)^{1+\delta}$$
 holds for all $r>r_0$ outside $E_{\delta}.$  
 \end{lemma}
 \begin{proof} The conclusion is clearly true  for $u\equiv0.$ Next, we assume that $u\not\equiv0.$
 Since $u\geq0$ is a non-decreasing  function,  there is $r_1>r_0$ such that $u(r_1)>0.$  Using the non-decreasing property of $u,$    
$\eta:=\lim_{r\to\infty}u(r)$
exists or $\eta=\infty.$  If $\eta=\infty,$ then $\eta^{-1}=0.$ 
 Set  
 $$E_\delta=\left\{r\in(r_0,\infty):  u'(r)>u(r)^{1+\delta}\right\}.$$
Since $u$ is  non-decreasing  on $(r_0, \infty),$    $u'(r)$ exists for  almost all  $r\in(r_0, \infty).$   It is therefore  
   \begin{eqnarray*}
\int_{E_\delta}dr 
 &\leq& \int_{r_0}^{r_1}dr+\int_{r_1}^\infty\frac{u'(r)}{u(r)^{1+\delta}}dr \\
 &=&\frac{1}{\delta u(r_1)^\delta}-\frac{1}{\delta \eta^\delta}+r_1-r_0 \\
 &<&\infty.
    \end{eqnarray*}
 This completes the proof. 
 \end{proof}

\begin{theorem}[Calculus Lemma]\label{cal} Let $k\geq0$ be a locally integrable function on $M.$ Assume that $k$ is locally bounded at $o.$ Then there exists a constant $C>0$ such that 
 for any   $\delta>0,$  there exists   a subset $E_{\delta}\subset(0,\infty)$ of finite Lebesgue measure such that
$$\int_{\partial\Delta(r)}kd\pi_r\leq C\Xi(r,\delta, \kappa)
    \left(\int_{\Delta(r)}g_r(o,x)kdv\right)^{(1+\delta)^2}$$
holds for all $r>0$ outside $E_{\delta},$ where $\Xi(r,\delta, \kappa)$ is defined by $(\ref{XI}).$  
 \end{theorem}
\begin{proof} 
  Lemma \ref{grest} yields  that
      \begin{eqnarray*}
\Lambda(r)&:=&   \int_{\Delta(r)}g_r(o,x)kdv \\
&=& \int_0^rdt\int_{\partial\Delta(t)}g_r(o,x)kd\sigma_t \\
      &=& 2A \int_{0}^rdt\int_0^\infty\frac{e^{-\frac{t^2}{as}}-e^{-\frac{r^2}{as}}}{V_{\max}(\sqrt{s})}ds\int_{\partial\Delta(t)}kd\sigma_t. 
      \end{eqnarray*}
     Differentiating $\Lambda,$  we obtain  
                 \begin{eqnarray*}
      \Lambda'(r)&=&\frac{d\Lambda(r)}{dr} \\
      &=&\frac{4Ar}{a}\int_0^\infty\frac{e^{-\frac{r^2}{at}}}{V_{\max}(\sqrt{t})}\frac{dt}{t}\int_{0}^rdt\int_{\partial\Delta(t)}kd\sigma_t. 
                       \end{eqnarray*}
      In further  
      $$\frac{d}{dr}\left(\frac{\Lambda'(r)}{r\displaystyle\int_0^\infty\frac{e^{-\frac{r^2}{at}}}{V_{\max}(\sqrt{t})}\frac{dt}{t}}\right)=\frac{4A}{a}\int_{\partial\Delta(r)}kd\sigma_r.$$
   Applying  Borel's lemma  to the left hand side on the above equality twice, then we see that  there exists a subset $F_\delta\subset(0,\infty)$ of finite Lebesgue measure such that   
            \begin{eqnarray*}\label{w1}
    &&  \int_{\partial\Delta(r)}kd\sigma_r   \\
    &\leq& 
      \frac{a}{4A}\left(r\int_0^\infty\frac{e^{-\frac{r^2}{at}}}{V_{\max}(\sqrt{t})}\frac{dt}{t}\right)^{-(1+\delta)}\Lambda(r)^{(1+\delta)^2}  \nonumber \\
      &=&\frac{a}{4A}\left(r\int_0^\infty\frac{e^{-\frac{r^2}{at}}}{V_{\max}(\sqrt{t})}\frac{dt}{t}\right)^{-(1+\delta)}\left(\int_{\Delta(r)}g_r(o,x)kdv\right)^{(1+\delta)^2} \nonumber
           \end{eqnarray*}
  holds for all $r>0$ outside $F_\delta.$
         By    Corollary \ref{pop},    there exists  a constant   $c>0$ such that  
 \begin{equation*}\label{w2}
 d\pi_r\leq c\left(|\kappa|+r^{-1}\right)\int_0^\infty\frac{e^{-\frac{r^2}{bt}}}{V_{\rm min}(\sqrt{t})}dt\cdot d\sigma_r
 \end{equation*}
holds for all $r>r_0,$ where $r_0>0$ is number such that $K\subset\Delta(r).$     Combining the above,   we conclude that   
  \begin{eqnarray*}\label{formula}
       \int_{\partial\Delta(r)}kd\pi_r  &\leq&  
     \frac{ac}{4A}\frac{\left(|\kappa|+r^{-1}\right)\displaystyle\int_0^\infty\frac{e^{-\frac{r^2}{bt}}}{V_{\rm min}(\sqrt{t})}dt}{\left(r\displaystyle\int_0^\infty\frac{e^{-\frac{r^2}{at}}}{V_{\max}(\sqrt{t})}\frac{dt}{t}\right)^{1+\delta}} \left(\int_{\Delta(r)}g_r(o,x)kdv\right)^{(1+\delta)^2}  \nonumber \\
     &=&  C\Xi(r,\delta, \kappa)
     \left(\int_{\Delta(r)}g_r(o,x)kdv\right)^{(1+\delta)^2}
            \end{eqnarray*}
             for $r>0$ outside $E_\delta:=F_\delta\cup(0, r_0],$ where $C=ac/4A.$ 
\end{proof}

\subsection{Logarithmic Derivative Lemma}~

Let $\psi$ be a meromorphic function on $M.$
The norm of the gradient $\nabla\psi$  is given as 
$$\|\nabla\psi\|^2=2\sum_{i,j=1}^m g^{i\overline j}\frac{\partial\psi}{\partial z_i}\overline{\frac{\partial \psi}{\partial  z_j}}$$
in   holomorphic local coordinates $z_1,\cdots,z_m,$ here $(g^{i \bar j})$ is the inverse of $(g_{i \bar j}).$
Define   
$T(r,\psi):=m(r,\psi)+N(r,\psi),$ where 
\begin{eqnarray*}
m(r,\psi)&=&\int_{\partial\Delta(r)}\log^+|\psi|d\pi_r, \\
N(r,\psi)&=& \frac{\pi^m}{(m-1)!}\int_{\psi^*\infty\cap \Delta(r)}g_r(o,x)\alpha^{m-1}.
 \end{eqnarray*}

   On $\mathbb P^1(\mathbb C),$ we take a singular metric
$$\Psi=\frac{1}{|\zeta|^2(1+\log^2|\zeta|)}\frac{\sqrt{-1}}{4\pi^2}d\zeta\wedge d\bar \zeta,$$
which satisfies   that 
$$\int_{\mathbb P^1(\mathbb C)}\Psi=1.$$

   We  first prove  two lemmas as follows. 
\begin{lemma}\label{oo12} Let $\psi\not\equiv0$ be a  meromorphic function on  $M.$  Then 
$$\frac{1}{4\pi}\int_{\Delta(r)}g_r(o,x)\frac{\|\nabla\psi\|^2}{|\psi|^2(1+\log^2|\psi|)}dv\leq T(r,\psi)+O(1).$$
\end{lemma}
\begin{proof}  Note that  
$$\frac{\|\nabla\psi\|^2}{|\psi|^2(1+\log^2|\psi|)}=4m\pi\frac{\psi^*\Psi\wedge\alpha^{m-1}}{\alpha^m}.$$
 Using     
 Fubini's theorem, we conclude that  
\begin{eqnarray*}
&& \frac{1}{4\pi}\int_{\Delta(r)}g_r(o,x)\frac{\|\nabla\psi\|^2}{|\psi|^2(1+\log^2|\psi|)}dv \\ 
&=&m\int_{\Delta(r)}g_r(o,x)\frac{\psi^*\Psi\wedge\alpha^{m-1}}{\alpha^m}dv  \\
&=&\frac{\pi^m}{(m-1)!}\int_{\mathbb P^1(\mathbb C)}\Psi(\zeta)\int_{\psi^*\zeta\cap \Delta(r)}g_r(o,x)\alpha^{m-1} \\
&=&\int_{\mathbb P^1(\mathbb C)}N\Big(r, \frac{1}{\psi-\zeta}\Big)\Psi(\zeta)  \\
&\leq&\int_{\mathbb P^1(\mathbb C)}\big{(}T(r,\psi)+O(1)\big{)}\Psi \\
&=& T(r,\psi)+O(1). 
\end{eqnarray*}
\end{proof}

\begin{lemma}\label{999a}  Let
$\psi\not\equiv0$ be a  meromorphic function on  $M.$  Then for any   $\delta>0,$ there exists a subset 
 $E_\delta\subset(0,\infty)$ of finite Lebesgue measure such that
\begin{eqnarray*}
  &&  \int_{\partial\Delta(r)}\log^+\frac{\|\nabla\psi\|^2}{|\psi|^2(1+\log^2|\psi|)}d\pi_r  \\
   &\leq& (1+\delta)^2 \log^+T(r,\psi)+\log^+ \Xi(r,\delta, \kappa)+O(1)
\end{eqnarray*}
 holds for all  $r>0$ outside  $E_\delta,$ where $\Xi(r,\delta, \kappa)$ is defined by $(\ref{XI}).$ 
\end{lemma}
\begin{proof} The concavity of $``\log"$ implies that    
\begin{eqnarray*}
  \int_{\partial\Delta(r)}\log^+\frac{\|\nabla\psi\|^2}{|\psi|^2(1+\log^2|\psi|)}d\pi_r  
    &\leq&  \log^+\int_{\partial\Delta(r)}\frac{\|\nabla\psi\|^2}{|\psi|^2(1+\log^2|\psi|)}d\pi_r+O(1).
\end{eqnarray*}
 Apply  Theorem \ref{cal} and Lemma \ref{oo12} to the first term on the right hand side of the above inequality, 
then  for any $\delta>0,$ there exists a subset 
 $E_\delta\subset(0,\infty)$ of finite Lebesgue measure such that  this term is bounded from above by 
    $$(1+\delta)^2 \log^+T(r,\psi)
    +\log^+ \Xi(r,\delta, \kappa) +O(1)$$
  for  $r>0$ outside  $E_\delta.$ This completes the proof. 
\end{proof}
Define
$$m\left(r,\frac{\|\nabla\psi\|}{|\psi|}\right)=\int_{\partial\Delta(r)}\log^+\frac{\|\nabla\psi\|}{|\psi|}d\pi_r.$$

\begin{theorem}[Logarithmic Derivative  Lemma]\label{log} Let
$\psi\not\equiv0$ be a  meromorphic function on  $M.$   Then for any   $\delta>0,$ there exists a  subset  $E_\delta\subset(0,\infty)$ of  finite Lebesgue measure such that 
\begin{eqnarray*}
   m\Big(r,\frac{\|\nabla\psi\|}{|\psi|}\Big)&\leq& \frac{2+(1+\delta)^2}{2}\log^+ T(r,\psi) 
   +\frac{1}{2}\log^+\Xi(r,\delta, \kappa)+O(1)
\end{eqnarray*}
 holds for all  $r>0$ outside  $E_\delta.$
\end{theorem}
\begin{proof}  We have 
\begin{eqnarray*}
    m\left(r,\frac{\|\nabla\psi\|}{|\psi|}\right)  
   &\leq& \frac{1}{2}\int_{\partial\Delta(r)}\log^+\frac{\|\nabla\psi\|^2}{|\psi|^2(1+\log^2|\psi|)}d\pi_r \ \ \  \  \    \   \  \  \    \    \   \   \\ 
 &&   +\frac{1}{2}\int_{\partial\Delta(r)}\log\left(1+\log^2|\psi|\right)d\pi_r \\
        &\leq&  \frac{1}{2}\int_{\partial\Delta(r)}\log^+\frac{\|\nabla\psi\|^2}{|\psi|^2(1+\log^2|\psi|)}d\pi_r  \\ 
   && +\log\int_{\partial\Delta(r)}\Big{(}\log^+|\psi|+\log^+\frac{1}{|\psi|}\Big{)}d\pi_r +O(1)  \\
   &\leq& \frac{1}{2}\int_{\partial\Delta(r)}\log^+\frac{\|\nabla\psi\|^2}{|\psi|^2(1+\log^2|\psi|)}d\pi_r+\log^+T(r,\psi)+O(1).
\end{eqnarray*}
Apply  Lemma \ref{999a} again,  we have the theorem proved.
   \end{proof}

\section{Proof of Theorem \ref{main1}}

Let $D=\sum_j\mu_jD_j$ be a divisor, where  $D_j$ are prime divisors.  The \emph{reduced form} of $D$ is defined by ${\rm Red}(D):=\sum_jD_j.$

\emph{Proof of Theorem $\ref{main1}$}

Write $D=D_1+\cdots+D_q$ as  the irreducible decomposition of $D.$
Equipping every holomorphic line bundle $\mathscr O(D_j)$ with a Hermitian
metric $h_j$ such that it  induces  the  Hermitian metric $h=h_1\otimes\cdots\otimes h_q$ on $L.$ 
Pick $s_j\in H^0(X, \mathscr O(D_j)$
such that  $(s_j)=D_j$ and $\|s_j\|<1.$
On $X,$ define a singular volume form
$$
  \Phi=\frac{\wedge^nc_1(L,h)}{\prod_{j=1}^q\|s_j\|^2}.$$
Set
$$f^*\Phi\wedge\alpha^{m-n}=\xi\alpha^m.$$
It is not hard to deduce     
$$dd^c[\log\xi]\geq f^*c_1(L, h_L)-f^*{\rm{Ric}}(\Omega)+\mathscr{R}-\left[{\rm{Red}}(f^*D)\right]$$
in the sense of currents.  
Hence, it  yields   that
\begin{eqnarray}\label{5q}
&& \frac{1}{4}\int_{\Delta(r)}g_r(o,x)\Delta\log\xi dv \\
&\geq& T_{f}(r,L)+T_{f}(r,K_X)+T(r,\mathscr{R})-\overline{N}_{f}(r,D).  \nonumber
\end{eqnarray}
 \ \ \ \   Since  $D$   is of   simple normal crossing type,  there exist  
  a finite  open covering $\{U_\lambda\}$ of $X,$ and finitely many   rational functions
$w_{\lambda1},\cdots,w_{\lambda n}$ on $X$ for each $\lambda,$  such that $w_{\lambda1},\cdots, w_{\lambda n}$ are holomorphic on $U_\lambda$  with    
\begin{eqnarray*}
  dw_{\lambda1}\wedge\cdots\wedge dw_{\lambda n}(x)\neq0, & & \ \  ^\forall x\in U_{\lambda}; \\
  D\cap U_{\lambda}=\big{\{}w_{\lambda1}\cdots w_{\lambda h_\lambda}=0\big{\}}, && \ \ ^\exists h_{\lambda}\leq n.
\end{eqnarray*}
In addition, we  can require that  $\mathscr O(D_j)|_{U_\lambda}\cong U_\lambda\times \mathbb C$ for 
all $\lambda,j.$ On  $U_\lambda,$   write 
$$\Phi=\frac{e_\lambda}{|w_{\lambda1}|^2\cdots|w_{\lambda h_{\lambda}}|^2}
\bigwedge_{k=1}^n\frac{\sqrt{-1}}{\pi}dw_{\lambda k}\wedge d\bar w_{\lambda k},$$
where  $e_\lambda$ is a  positive smooth function on $U_\lambda.$  
Let  $\{\phi_\lambda\}$ be a partition of the unity subordinate to $\{U_\lambda\}.$ Set 
$$\Phi_\lambda=\frac{\phi_\lambda e_\lambda}{|w_{\lambda1}|^2\cdots|w_{\lambda h_{\lambda}}|^2}
\bigwedge_{k=1}^n\frac{\sqrt{-1}}{\pi}dw_{\lambda k}\wedge d\bar w_{\lambda k}.$$
 Again, put $f_{\lambda k}=w_{\lambda k}\circ f.$  On  $f^{-1}(U_\lambda),$ we have  
\begin{eqnarray*}
 f^*\Phi_\lambda&=&
   \frac{\phi_{\lambda}\circ f\cdot e_\lambda\circ f}{|f_{\lambda1}|^2\cdots|f_{\lambda h_{\lambda}}|^2}
   \bigwedge_{k=1}^n\frac{\sqrt{-1}}{\pi}df_{\lambda k}\wedge d\bar f_{\lambda k} \\
   &=& \phi_{\lambda}\circ f\cdot e_\lambda\circ f\sum_{1\leq i_1\not=\cdots\not= i_n\leq m}
   \frac{\Big|\frac{\partial f_{\lambda1}}{\partial z_{i_1}}\Big|^2}{|f_{\lambda 1}|^2}\cdots 
   \frac{\Big|\frac{\partial f_{\lambda h_\lambda}}{\partial z_{i_{h_\lambda}}}\Big|^2}{|f_{\lambda h_\lambda}|^2}
   \left|\frac{\partial f_{\lambda (h_\lambda+1)}}{\partial z_{i_{h_\lambda+1}}}\right|^2 \\
 &&  \cdots\left|\frac{\partial f_{\lambda n}}{\partial z_{i_{n}}}\right|^2 
    \Big(\frac{\sqrt{-1}}{\pi}\Big)^ndz_{i_1}\wedge d\bar z_{i_1}\wedge\cdots\wedge dz_{i_n}\wedge d\bar z_{i_n}.
\end{eqnarray*}
 Fix any  $x_0\in M,$  one can  take  holomorphic local  coordinates $z_1,\cdots,z_m$ near $x_0$ and   holomorphic local  coordinates
  $\zeta_1,\cdots,\zeta_n$ 
near $f(x_0)$ such that
$$ \alpha|_{x_0}=\frac{\sqrt{-1}}{\pi}\sum_{j=1}^m dz_j\wedge d\bar{z}_j$$
and 
$$ c_1(L, h)\big|_{f(x_0)}=\frac{\sqrt{-1}}{\pi}\sum_{j=1}^n d\zeta_j\wedge d\bar{\zeta}_j.$$
 Set   
$$f^*\Phi_\lambda\wedge\alpha^{m-n}=\xi_\lambda\alpha^m.$$
Then, we have  $$\xi=\sum_\lambda \xi_\lambda$$ and   
 \begin{eqnarray*}
&& \xi_\lambda|_{x_0} \\ 
&=& \phi_{\lambda}\circ f\cdot e_\lambda\circ f\sum_{1\leq i_1\not=\cdots\not= i_n\leq m}
   \frac{\Big|\frac{\partial f_{\lambda1}}{\partial z_{i_1}}\Big|^2}{|f_{\lambda 1}|^2}\cdots 
   \frac{\Big|\frac{\partial f_{\lambda h_\lambda}}{\partial z_{i_{h_\lambda}}}\Big|^2}{|f_{\lambda h_\lambda}|^2}
   \left|\frac{\partial f_{\lambda (h_\lambda+1)}}{\partial z_{i_{h_\lambda+1}}}\right|^2\cdots\left|\frac{\partial f_{\lambda n}}{\partial z_{i_{n}}}\right|^2 \\
   &\leq&  \phi_{\lambda}\circ f\cdot e_\lambda\circ f\sum_{1\leq i_1\not=\cdots\not= i_n\leq m}
    \frac{\big\|\nabla f_{\lambda1}\big\|^2}{|f_{\lambda 1}|^2}\cdots 
   \frac{\big\|\nabla f_{\lambda h_\lambda}\big\|^2}{|f_{\lambda h_\lambda}|^2} \\
   &&\cdot \big\|\nabla f_{\lambda(h_\lambda+1)}\big\|^2\cdots\big\|\nabla f_{\lambda n}\big\|^2.
\end{eqnarray*} 
Define a non-negative function $\varrho$ on $M$ by  
\begin{equation}\label{wer}
  f^*c_1(L, h)\wedge\alpha^{m-1}=\varrho\alpha^m.
\end{equation}
Again, put  $f_j=\zeta_j\circ f$ with  $j=1,\cdots,n.$  Then, we have      
\begin{equation*}
    f^*c_1(L, h)\wedge\alpha^{m-1}\big|_{x_0}=\frac{(m-1)!}{2}\sum_{j=1}^m\big\|\nabla f_j\big\|^2\alpha^m, 
\end{equation*}
 which yields that   
$$\varrho|_{x_0}=(m-1)!\sum_{i=1}^n\sum_{j=1}^m\Big|\frac{\partial f_i}{\partial z_j}\Big|^2
=\frac{(m-1)!}{2}\sum_{j=1}^n\big\|\nabla f_j\big\|^2.$$
 Put together   the above, we are led to 
$$\xi_\lambda\leq 
\frac{ \phi_{\lambda}\circ f\cdot e_\lambda\circ f\cdot(2\varrho)^{n-h_\lambda}}{(m-1)!^{n-h_\lambda}}\sum_{1\leq i_1\not=\cdots\not= i_n\leq m}
    \frac{\big\|\nabla f_{\lambda1}\big\|^2}{|f_{\lambda 1}|^2}\cdots 
   \frac{\big\|\nabla f_{\lambda h_\lambda}\big\|^2}{|f_{\lambda h_\lambda}|^2}
$$
on $f^{-1}(U_\lambda).$
Since $\phi_\lambda\circ f\cdot e_\lambda\circ f$ is bounded on $M$ and   
$$\log^+\xi\leq \sum_\lambda\log^+\xi_\lambda+O(1),$$
 we obtain  
\begin{equation}\label{bbd}
   \log^+\xi\leq O\left(\log^+\varrho+\sum_{k, \lambda}\log^+\frac{\|\nabla f_{\lambda k}\|}{|f_{\lambda k}|}+1\right). 
 \end{equation}  
  Jensen-Dynkin formula yields  
\begin{equation*}
 \frac{1}{2}\int_{\Delta(r)}g_r(o,x)\Delta\log\xi dv
=\int_{\partial\Delta(r)}\log\xi d\pi_r+O(1).
\end{equation*}
Combining this with (\ref{bbd}) and Theorem \ref{log} to get  
\begin{eqnarray*}
&& \frac{1}{4}\int_{\Delta(r)}g_r(o,x)\Delta\log\xi dv\\
   &\leq& O\left(\sum_{k,\lambda}m\left(r,\frac{\|\nabla f_{\lambda k}\|}{|f_{\lambda k}|}\right)+\log^+\int_{\partial\Delta(r)}\varrho d\pi_r+1\right) \\
   &\leq& O\left(\sum_{k,\lambda}\log^+ T(r,f_{\lambda k})
   +\log^+ \Xi(r, \delta,\kappa)+\log^+\int_{\partial\Delta(r)}\varrho d\pi_r\right)
   \\
      &\leq& O\left(\log^+ T_f(r,L)+\log^+\Xi(r,\delta,\kappa)
   +\log^+\int_{\partial\Delta(r)}\varrho d\pi_r\right).   
   \end{eqnarray*}
Using Theorem \ref{cal} and (\ref{wer}),  for any $\delta>0,$ there exists a subset $E_\delta\subset(0,\infty)$ of finite Lebesgue measure such that 
   \begin{eqnarray*}
 \log^+\int_{\partial\Delta(r)}\varrho d\pi_r 
   &\leq&  (1+\delta)^2\log^+ T_f(r,L)
    +\log^+\Xi(r,\delta,\kappa)+O(1)
      \end{eqnarray*}
   holds for all $r>0$ outside $E_\delta.$  
It is therefore   
\begin{eqnarray*}\label{6q}
     \frac{1}{4}\int_{\Delta(r)}g_r(o,x)\Delta\log\xi dv 
 &\leq&  O\left(\log^+ T_f(r,L)+
    \log\Xi(r, \delta,\kappa)
   \right) 
\end{eqnarray*}
holds for all $r>0$ outside $E_\delta.$  By this with   (\ref{5q}),  we prove the theorem. $\square$

      \section{Proofs of  Theorems  \ref{main2}, \ref{main3} and \ref{main4}}
      \vskip\baselineskip

       \subsection{Estimates of  two Infinite Integrals}~

 We treat   a non-parabolic complete non-compact K\"ahler manifold $M$ with  non-negative Ricci curvature. 
Fix a reference point $o\in M.$
Let  $V(r)$ denote   the Riemannian volume of   geodesic ball  $B(r)$ centered at  $o$ with radius $r$ in $M.$  
We shall prove that for any $\mu>0,$ there exist  constants $c_\mu, C_\mu>0$ such that for all $r>0$
 $$\int_0^\infty\frac{e^{-\frac{r^2}{\mu t}}}{V(\sqrt{t})}\frac{dt}{t}\geq \frac{c_\mu}{V(r)}$$
 and for all $r>1$
   $$\int_0^\infty\frac{e^{-\frac{r^2}{\mu t}}}{V(\sqrt{t})}dt\leq \frac{C_\mu r^2}{V(r)}+2\int_{r}^\infty\frac{tdt}{V(t)}.$$
  In doing so, one needs to  introduce the classical volume comparison theorem  by Bishop-Gromov  (see \cite{B, S-Y}). 

Let $N$ be a complete $n$-dimensional Riemannian manifold. Fix a reference point $o\in N.$  We  use  $V(r)$ to denote the Riemannian volume of the geodesic ball centered at $o$ with radius $r$ in $N.$  
Also,  let  $N^K$ be a (simply-connected) $n$-dimensional space form  with constant sectional curvature $K.$
  We use   $V(K, r)$ to denote the Riemannian volume of a geodesic ball with radius $r$ in $N^K.$  
  
\begin{lemma}[Volume Comparison Theorem]\label{comp}   If  ${\rm{Ric}}_N\geq(n-1)K$ for a constant $K,$ then the volume ratio $V(r)/V(K, r)$ is non-increasing in $r>0,$ and it tends to $1$ as $r\to0.$  Hence, 
we have 
$V(r)\leq V(K, r)$
 for all  $r\geq0.$
\end{lemma}

  On the other hand,  the non-parabolicity of $M$ implies that   
$$\int_0^\infty\frac{tdt}{V(\sqrt{t})}<\infty.$$
Hence, we have 
\begin{equation}\label{dongg}
r^2=o(V(r)), \ \ \ \ r\to\infty.
\end{equation}

  \begin{theorem}\label{est1} For any number $\mu>0,$ there exists a constant $c_\mu>0$ such that
  $$\int_0^\infty\frac{e^{-\frac{r^2}{\mu t}}}{V(\sqrt{t})}\frac{dt}{t}\geq \frac{c_\mu}{V(r)}$$ 
  holds for all $r>0.$
  \end{theorem} 
 \begin{proof} Note that 
 $$\int_0^\infty\frac{e^{-\frac{r^2}{\mu t}}}{V(\sqrt{t})}\frac{dt}{t}\geq \int_{r^2}^\infty\frac{e^{-\frac{r^2}{\mu t}}}{V(\sqrt{t})}\frac{dt}{t}.$$
 When $t\geq r^2,$  we can employ Lemma \ref{comp} (compared with $\mathbb C^m$) to obtain  
 $$\frac{V(r)}{V(\sqrt t)}\geq\left(\frac{r}{\sqrt{t}}\right)^{2m}=\frac{r^{2m}}{t^m}.$$
 That is, 
 $$V(\sqrt t)\leq \frac{t^{m}}{r^{2m}}V(r).$$
 Hence, we conclude that
 \begin{eqnarray*}
\int_{r^2}^\infty\frac{e^{-\frac{r^2}{\mu t}}}{V(\sqrt{t})}\frac{dt}{t}&\geq&  \int_{r^2}^\infty\frac{e^{-\frac{r^2}{\mu t}}}{V(r)}\frac{r^{2m}}{t^m}\frac{dt}{t} \\
&=&\frac{r^{2m}}{V(r)} \int_{r^2}^\infty e^{-\frac{r^2}{\mu t}}\frac{dt}{t^{m+1}} \\
&\geq& e^{-\frac{1}{\mu}}\frac{r^{2m}}{V(r)} \int_{r^2}^\infty\frac{dt}{t^{m+1}}  \\
&=& \frac{m^{-1}e^{-\frac{1}{\mu}}}{V(r)}.
 \end{eqnarray*}
 This gives   
 $$\int_0^\infty\frac{e^{-\frac{r^2}{\mu t}}}{V(\sqrt{t})}\frac{dt}{t}\geq \frac{m^{-1}e^{-\frac{1}{\mu}}}{V(r)}.$$
Hence, we have the desired result by setting $c_\mu=m^{-1}e^{-\frac{1}{\mu}}.$
    \end{proof}

  \begin{theorem}\label{est2} For any number $\mu>0,$ there exists a constant $C_\mu>0$ such that
  $$\int_0^\infty\frac{e^{-\frac{r^2}{\mu t}}}{V(\sqrt{t})}dt\leq \frac{C_\mu r^2}{V(r)}+2\int_{r}^\infty\frac{tdt}{V(t)}$$
  holds for all $r>1.$
  \end{theorem} 
\begin{proof}
 For $r>0,$ we have  
 \begin{eqnarray*}
 \int_0^\infty\frac{e^{-\frac{r^2}{\mu t}}}{V(\sqrt{t})}dt&=& \int_0^{r^2}\frac{e^{-\frac{r^2}{\mu t}}}{V(\sqrt{t})}dt
 +\int_{r^2}^\infty\frac{e^{-\frac{r^2}{\mu t}}}{V(\sqrt{t})}dt \\
 &\leq & \int_0^{r^2}\frac{e^{-\frac{r^2}{\mu t}}}{V(\sqrt{t})}dt
 +\int_{r^2}^\infty\frac{dt}{V(\sqrt{t})} \\
 &=& \int_0^{r^2}\frac{e^{-\frac{r^2}{\mu t}}}{V(\sqrt{t})}dt+2\int_{r}^\infty\frac{tdt}{V(t)}.
  \end{eqnarray*}
  Set
  $$I=\int_0^{r^2}\frac{e^{-\frac{r^2}{\mu t}}}{V(\sqrt{t})}dt.$$
Putting  $s=r^4/t$ to get     
 \begin{eqnarray*}
I &=& r^4\int_{r^2}^\infty\frac{e^{-\frac{s}{\mu r^2}}}{V(r^2/\sqrt{s})}\frac{ds}{s^2}.
 \end{eqnarray*}
Since $s\geq r^2,$ it yields from  Lemma \ref{comp} that    
 $$\frac{V(r^2/\sqrt{s})}{V(\sqrt s)}=\frac{V(\sqrt t)}{V(\sqrt s)}\geq\left(\frac{\sqrt t}{\sqrt{s}}\right)^{2m}=\frac{t^{m}}{s^m}=\frac{r^{4m}}{s^{2m}}.$$
That is,  
$$V\left(\frac{r^2}{\sqrt{s}}\right)=V(\sqrt t)\geq \frac{r^{4m}}{s^{2m}}V(\sqrt s).$$
Thus, we have 
 \begin{eqnarray*}
I \leq\frac{1}{r^{4m-4}}\int_{r^2}^\infty \frac{s^{2m-2}e^{-\frac{s}{\mu r^2}}}{V(\sqrt{s})}ds 
\leq \frac{1}{r^{4m-4}V(r)}\int_{r^2}^\infty s^{2m-2}e^{-\frac{s}{\mu r^2}}ds.
 \end{eqnarray*}
 By integration-by-parts formula, it is not difficult to deduce  that   there exists a constant  $C_\mu>0$ such that 
 $$\int_{r^2}^\infty s^{2m-2}e^{-\frac{s}{\mu r^2}}ds\leq C_\mu r^{4m-2}$$ 
holds for all $r>1.$  Hence,  we have   
 $I\leq C_\mu r^2/V(r)$
for $r>1.$ This completes the proof.
\end{proof}
      
  \subsection{Proofs}~

\emph{Proof of Theorem $\ref{main2}$}

Since $e^{-r^2/bt}$ is  decreasing in $r>0,$
we have 
$$\int_0^\infty\frac{e^{-\frac{r^2}{bt}}}{V_{\rm min}(\sqrt{t})}dt\leq C_0$$
for some constant $C_0>0.$ Since all $M_j^,s$ have  Ricci curvature bounded from below by a constant,  the Ricci curvature of $M$   can be  bounded from below by a  constant $\kappa_0<0.$  Denote by  $V(\kappa_0, r)$  the Riemannian volume of a geodesic ball with radius $r$ in the space form of dimension $2m$ with constant sectional  $\kappa_0.$ It yields from  Lemma \ref{comp} that
      \begin{eqnarray*} 
V(r)&\leq& V(\kappa_0, r) \\
&=&\omega_{2m-1}\int_0^r\left(\frac{\sinh\sqrt{-\kappa_0}}{\sqrt{-\kappa_0}}\right)^{2m-1}dt,
      \end{eqnarray*}
where $\omega_{2m-1}$ is the Riemannian area of the unit ball in $\mathbb C^{m}.$ By  
      \begin{eqnarray*}
\int_0^r\left(\frac{\sinh\sqrt{-\kappa_0}}{\sqrt{-\kappa_0}}\right)^{2m-1}dt&\leq& \int_0^r\left(\frac{e^{\sqrt{-\kappa_0}t}}{\sqrt{-\kappa_0}}\right)^{2m-1}dt \\
&\leq& \frac{e^{(2m-1)\sqrt{-\kappa_0}r}}{(2m-1)(-\kappa_0)^m},
      \end{eqnarray*}
we obtain 
$$V_{\max}(r)\leq V(r)\leq  \frac{\omega_{2m-1}}{(2m-1)(-\kappa_0)^m}e^{(2m-1)\sqrt{-\kappa_0}r}.$$
Thus, we are led to
$$ \Xi(r, \delta, \kappa)\leq \frac{C_0\left(|\kappa|+r^{-1}\right)}{\left((2m-1)(-\kappa_0)^mr\displaystyle\int_0^\infty e^{-\frac{r^2}{at}-(2m-1)\sqrt{-\kappa_0 t}}\frac{dt}{t}\right)^{1+\delta}}.  
$$
For $r>1,$ we have 
      \begin{eqnarray*}
&& \int_0^1 e^{-\frac{r^2}{at}-(2m-1)\sqrt{-\kappa_0 t}}\frac{dt}{t}  \\
&\geq& e^{-(2m-1)\sqrt{-\kappa_0}}\int_0^1 e^{-\frac{r^2}{at}}\frac{dt}{t} \\
&=& e^{-(2m-1)\sqrt{-\kappa_0}}\int_{r^2}^\infty e^{-\frac{t}{a}}\frac{dt}{t} \\
&=& e^{-(2m-1)\sqrt{-\kappa_0}}\left(ar^{-2}e^{-\frac{r^2}{a}}-\int_{r^2}^\infty e^{-\frac{t}{a}}\frac{dt}{t^2}\right) \\
&\geq& ae^{-(2m-1)\sqrt{-\kappa_0}}r^{-2}e^{-\frac{r^2}{a}}-e^{-(2m-1)\sqrt{-\kappa_0}}\int_{r^2}^\infty e^{-\frac{t}{a}}\frac{dt}{t} \\
&\geq& ae^{-(2m-1)\sqrt{-\kappa_0}}r^{-2}e^{-\frac{r}{a}}-\int_0^1 e^{-\frac{r^2}{at}-(2m-1)\sqrt{-\kappa_0 t}}\frac{dt}{t}. 
      \end{eqnarray*}
That is,     for $r>1$
      \begin{eqnarray*}
       \int_0^\infty e^{-\frac{r^2}{at}-(2m-1)\sqrt{-\kappa_0 t}}\frac{dt}{t}  
   &\geq&\int_0^1 e^{-\frac{r^2}{at}-(2m-1)\sqrt{-\kappa_0 t}}\frac{dt}{t} \\
   &\geq& \frac{a}{2}e^{-(2m-1)\sqrt{-\kappa_0}}r^{-2}e^{-\frac{r^2}{a}}. 
         \end{eqnarray*}
 Combining the above, we conclude that  
$$\log^+\Xi(r, \delta, \kappa)\leq O(r^2)$$ 
as $r\to\infty.$ Thus, we have the theorem proved due to Theorem \ref{main1}.

\emph{Proof of Theorem $\ref{main3}$}

Since  all $M_j^,s$ have non-negative Ricci curvature, we can take   $\kappa=0.$ 
Note  that  the central part  $K$ of $M$ is compact, 
 using Theorems \ref{est1} and \ref{est2}, there exist constants $c_1, c_2>0$ such that  
  $$\int_0^\infty\frac{e^{-\frac{r^2}{a t}}}{V_j(\sqrt{t})}\frac{dt}{t}\geq \frac{c_1}{V_j(r)}, \ \ \ \    j=1,\cdots,\vartheta$$ 
holds   for all $r>0,$ and  
      $$\int_0^\infty\frac{e^{-\frac{r^2}{b t}}}{V_j(\sqrt{t})}dt\leq \frac{c_2 r^2}{V_j(r)}+2\int_{r}^\infty\frac{tdt}{V_j(t)}, \ \ \ \    j=1,\cdots,\vartheta$$
holds   for all $r>1.$ Again,  since  $r^2=o(V_j(r))$ and $V_j(r)\leq O(r^{2m})$ as $r\to\infty$ due to  Lemma \ref{comp} and (\ref{dongg}), 
  we  obtain    
  \begin{equation}\label{dd1}
\Xi(r,\delta, \kappa)\leq \frac{\frac{c_2 r}{V_{\min}(r)}+2r^{-1}\displaystyle\int_{r}^\infty\frac{tdt}{V_{\min}(t)}}{\left(\frac{c_1r}{V_{\max}(r)}\right)^{1+\delta}} 
\leq O\left(r^{(2m-1)(1+\delta)}\right)
\end{equation}
as $r\to\infty.$ It is therefore 
$$\log^+\Xi(r,\delta, \kappa)\leq O(r)$$
as $r\to\infty.$  
By Theorem \ref{main1}, we have the theorem proved. 

\emph{Proof of Theorem $\ref{main4}$}

       Take $\kappa=0.$ Since  $M$ is  homogeneous,  we have  
  $$\Xi(r,\delta, \kappa)\simeq \frac{r^{-1}\displaystyle\int_0^\infty\frac{e^{-\frac{r^2}{bt}}}{V(\sqrt{t})}dt}{\left(r\displaystyle\int_0^\infty\frac{e^{-\frac{r^2}{at}}}{V(\sqrt{t})}\frac{dt}{t}\right)^{1+\delta}}.$$
By (\ref{dd1}), there is a constant $C>0$ such that 
      \begin{eqnarray*}
      \Xi(r,\delta, \kappa)&\leq& C\frac{\frac{r}{V(r)}+r^{-1}\displaystyle\int_{r}^\infty\frac{tdt}{V(t)}}{\left(\frac{r}{V(r)}\right)^{1+\delta}} \\
      &=&\left(\frac{V(r)}{r}\right)^{\delta}\left(1+\frac{V(r)}{r^2}\int_0^\infty\frac{tdt}{V(t)}\right),
            \end{eqnarray*}
      which leads  to 
              \begin{eqnarray*}
     \log^+\Xi(r,\delta,\kappa)  
     &\leq& \log^+\left(\frac{V(r)}{r^{2}}\int_r^\infty\frac{tdt}{V(t)}\right)+2\delta\log^+\frac{V(r)}{r}+O(1) \\
     &\leq& \log^+E(r)+(4m-2)\delta\log r+O(1).
      \end{eqnarray*}
      This proves the theorem by using Theorem \ref{main1}.

 \section{Examples}
 \vskip\baselineskip

\subsection{Hermitian    Manifolds with $(PH)$}~

Thanks to   Theorem \ref{equi1} and Theorem \ref{equi2}, if  $(PH)$ is satisfied for  complete weighted manifolds,
then   Property \ref{ap1} and Property \ref{ap2} can  be satisfied.

\begin{example}  \emph{Complete  Hermitian     manifolds with non-negative Ricci curvature}
    \begin{enumerate}
          \item[$\bullet$]  
Double volume property   follows  from  
   Bishop-Gromov's volume comparison theorem   (see Lemma \ref{comp});  Poincar\'e inequality 
 follows from the work  by P. Buser \cite{Bu}  (see 
\cite{Sa1}, Section 5.6.3 also);  Parabolic Harnack inequality  and two-sided heat kernel bounds were  obtained     by Li-Yau \cite{L-Y}. 
     \end{enumerate}
\end{example}

 \begin{example}   \emph{Convex domains in complex Euclidean spaces}  
    \begin{enumerate}
        \item[$\bullet$]  
The geodesic distance for convex domains  is the Euclidean distance (of course, the Neumann boundary condition  is  
assumed). 
   Poincar\'e inequality and double volume property are  well-known results.
Harnack inequality and two-sided  heat kernel bounds  
  can be derived  by the argument of 
  Li-Yau \cite{L-Y}.       \end{enumerate}
  \end{example} 
  
  \begin{example} 
  \emph{Complements of  convex domains in complex Euclidean spaces.}
    \begin{enumerate}
    \item[$\bullet$]  

Parabolic Harnack inequality was established   by Saloff-Coste-Gyrya \cite{S-C}. More generally, the parabolic Harnac inequality  holds on  inner uniform domains. 
We note that the complements of convex domains are inner uniform but the unbounded convex domains are not. 
     \end{enumerate}
\end{example}

  \begin{example}   \emph{Connected complex Lie groups with polynomial volume growth}  
    \begin{enumerate}
            \item[$\bullet$]  
  These  are connected complex  Lie groups such that for any compact neighborhood
$K$ of the origin,  we have $|K^n|\leq Cn^c$ for every  integer $n\geq1$ and  some constants $C, c>0,$ where 
  $K^n=\{g_1\cdots g_n: g_j\in K\}$ and $|K^n|$ is the Haar measure of $K^n.$ Note that nilpotent Lie groups are always of this type. 
  By a classical    result  of Y. Guivarc'h, there exists  a positive integer $N$ such that $c_1n^N\leq|K^n|\leq C_1n^N$ for every  integer $n\geq1$ and some constants $C_1, c_1>0.$ This gives 
   the double volume property for any left-invariant Hermitian metric.  Poincar\'e inequality and parabolic Harnack inequality are also obtained     in \cite{Sa1} and \cite{V-S-C}, respectively. Note that  double 
   volume property, Poincar\'e inequality and parabolic Harnack inequality also hold for the sub-Laplacians with the form
    $\Delta=\sum_{j=1}^kX_j^2$ associated with a family  of left-invariant vector fields $\mathcal F=\{X_1,\cdots, X_k\},$ as long as $\mathcal F$ generates the Lie algebra which is called the H\"ormander condition (see \cite{V-S-C}).  
      \end{enumerate}
\end{example} 
 
 \subsection{Complete K\"ahler  Connected Sums with $(PH)$}~

Let $M$ be a complete non-compact Riemannian  manifold. Let $K(x)$ be  the minimal sectional curvature of $M$ at $x\in M.$
Say that $M$ has \emph{asymptotically non-negative sectional curvature}, if there
 exist a reference point $o\in M$ and a continuous decreasing function $k>0$ on $(0,\infty)$ such that $K(x)\geq-k(\rho(x))$ holds for all $x\in M$ and 
$$\int_1^\infty tk(t)dt<\infty.$$
Such manifolds were studied in \cite{Ka1, L-T3}.  We say that $M$ satisfies the property  ``relative connectedness of the annuli"  if 
   \begin{enumerate}
   \item[$\bullet$]
$(RCA)$  \emph{Relative connectedness of the annuli}$:$ there exists $\alpha>0$ such that for all $r>0$ large enough  and 
 all $x,y\in \partial B(r),$  there is a   path  connecting $x$ and $y$  inside  the annulus 
$B(\alpha r)\setminus B(\alpha^{-1}r).$ 
   \end{enumerate}

  \begin{example}   \emph{Complete non-compact K\"ahler manifolds with 
  asymptotically non-negative sectional curvature}  
    \begin{enumerate}
            \item[$\bullet$]  
Such manifolds have a finite number of ends, and thus  can be written
as a connected sum  $M=M_1\#\cdots\#M_\vartheta$ of complete non-compact K\"ahler manifolds. Furthermore,
all $M_j^,s$  satisfy   double volume property,  parabolic Harnack inequality and relative connectedness  of the annuli. 
      \end{enumerate}
 \end{example}

  \begin{example}   \emph{Complete non-compact K\"ahler manifolds with non-negative Ricci curvature
outside a compact set}  
    \begin{enumerate}
            \item[$\bullet$]  
Such manifolds also have  finitely many ends (see \cite{Ca, L-T3}). 
 Hence,  it  can be written as  a connected sum $M=M_1\#\cdots\#M_\vartheta$ of complete non-compact K\"ahler manifolds.   
 Every  $M_j$ corresponds to one end of $M$
and  is  thought of as a manifold with non-negative Ricci curvature outside a compact set. 
 It is noted  that if one end $M_j$ satisfies relative connectedness of the annuli,  then it   satisfies  
 double volume property and Poincar\'e inequality (see \cite{S-C1}); furthermore,  it satisfies  parabolic Harnack inequality
 according to  Theorem \ref{equi2}.  
      \end{enumerate}
   \end{example}

\vskip\baselineskip

\noindent\textbf{Acknowledgements} The research work is supported by the Natural Science Foundation of Shandong Province of China
 ($\#$ZR202211290346).

\noindent\textbf{Conflict of Interest Statement.}  On behalf of all authors, the 
author states that there is no conflict of interest.

\vskip\baselineskip

\end{document}